\documentclass[reqno]{amsart}
\usepackage{mathrsfs}
\usepackage{amssymb}
\usepackage{latexsym}
\usepackage{yfonts}
\usepackage{natbib}
\usepackage{amssymb,amsmath,euscript}
\usepackage{color}

\usepackage{graphicx,color}

\usepackage[plainpages=false,colorlinks,hyperindex,bookmarksopen,linkcolor=red,citecolor=blue,urlcolor=blue]{hyperref}

\bibpunct{[}{]}{;}{n}{,}{,}

\newtheorem{tm}{Theorem}
\newtheorem{defin}{Definition}
\newtheorem{rk}{Remark}
\newtheorem{prop}{Proposition}
\newtheorem{lem}{Lemma}

\numberwithin{equation}{section}

\newcommand{\RR}[1]{\mathbb{#1}}

\newcommand{\rd}{{\mathbb R^d}}

\def\R{{\mathbb R}}

\def\l{{\langle}}
\def\r{\rangle}

\def\a{\alpha}

\def\eps{\varepsilon}

\def\E{{\mathbb E}}
\def\P{{\mathbb P}}

\allowdisplaybreaks

\begin{document}

\title{Space-time fractional stochastic partial differential equations}

\author{Jebessa B Mijena}
\address{Department of Mathematics, Georgia College \& State University, Milledgeville, GA 31061}
\email{jebessa.mijena@gcsu.edu}
\author{ Erkan Nane}
\address{Department of Mathematics and Statistics, Auburn University, Auburn, AL 36849, USA}
\email{nane@auburn.edu}

%\author{Yimin Xiao}
%\address{Department Statistics and Probability,
%Michigan State University, East Lansing, MI 48823}
%\email{xiaoyimi@stt.msu.edu}
%\thanks{Research partially
%supported by NSF grant DMS-1006903.}


\date{\today}

\subjclass[2000]{}

\begin{abstract}
We consider   non-linear time-fractional  stochastic heat type equation
$$\partial^\beta_tu_t(x)=-\nu(-\Delta)^{\alpha/2} u_t(x)+I^{1-\beta}_t[\sigma(u)\stackrel{\cdot}{W}(t,x)]$$ in $(d+1)$ dimensions, where $\nu>0, \beta\in (0,1)$, $\alpha\in (0,2]$ and $d<\min\{2,\beta^{-1}\}\a$,  $\partial^\beta_t$ is the Caputo fractional derivative, $-(-\Delta)^{\alpha/2} $ is the generator of an isotropic stable process, $I^{1-\beta}_t$ is the fractional integral operator,  $\stackrel{\cdot}{W}(t,x)$ is space-time white noise, and $\sigma:\RR{R}\to\RR{R}$ is Lipschitz continuous.
 Time fractional stochastic heat type equations might be used to model phenomenon with random effects with thermal memory. We prove existence and uniqueness of mild solutions to this equation and establish conditions under which the solution is continuous. Our results  extend the   results in the case of parabolic stochastic partial differential equations obtained in \cite{foondun-khoshnevisan-09, walsh}. In sharp contrast to the
   stochastic  partial differential equations studied earlier in \cite{foondun-khoshnevisan-09, khoshnevisan-cbms, walsh}, in some  cases our results give  existence of random field solutions in spatial dimensions $d=1,2,3$. Under faster than linear growth of $\sigma$, we show that time fractional stochastic partial differential equation   has no finite energy solution. This extends the result of Foondun and Parshad \cite{foondun-parshad} in the case of parabolic stochastic partial differential equations. We also establish a connection of the time fractional stochastic  partial differential equations to higher order parabolic stochastic differential equations.
\end{abstract}


\maketitle

\section{Introduction}
Stochastic partial differential equations (SPDE) have been studied  in  mathematics, and various sciences; see, for example, Khoshnevisan \cite{khoshnevisan-cbms} for a  big list of references. The area of SPDEs is interesting to mathematicians as it contains big number of hard open problems. SPDE's have  been applied in many disciplines that include applied mathematics, statistical mechanics, theoretical physics,  theoretical neuroscience, theory of complex chemical reactions, fluid dynamics, hydrology,  and mathematical finance.
In this paper we introduce new  time fractional stochastic partial differential equations (TSPDE) in the sense of Walsh \cite{walsh},  and prove existence and uniqueness of  mild solutions. We  establish the conditions under which the solution is continuous. When $\sigma $  grows faster than Lipschitz we show that  no finite energy solution exists.

 The study of the time-fractional diffusion equations has  recently attracted a lot of attention and a typical form of the time fractional equations is $\partial^\beta_tu=\Delta u$   where $\partial^\beta_t$ is the Caputo fractional derivative with $\beta\in (0,1)$ and $\Delta=\sum_{i=1}^d\partial^2_{x_i} $ is the Laplacian. These equations are related with anomalous diffusions or diffusions in non-homogeneous media, with random fractal structures; see, for instance, \cite{meerschaert-nane-xiao}.
The Caputo fractional derivative $\partial^\beta_t$  first appeared in \cite{Caputo} is defined for $0<\beta<1$ by
\begin{equation}\label{CaputoDef}
\partial^\beta_t u_t(x)=\frac{1}{\Gamma(1-\beta)}\int_0^t \partial_r
u_r(x)\frac{dr}{(t-r)^\beta} .
\end{equation}
Its Laplace transform
\begin{equation}\label{CaputolT}
\int_0^\infty e^{-st} \partial_t^\beta u_t(x) dt=s^\beta \tilde u_s(x)-s^{\beta-1} u_0(x),
\end{equation}
where $\tilde u_s(x) = \int_0^\infty e^{-st}u_t(x)dt$ and incorporates the initial value in the same way as the first
derivative.

Starting from the works  \cite{Koc89, Nig86, Wyss86, zaslavsky} much effort has been made in order to introduce a rigorous mathematical approach; see,  for example,  \cite{NANERW} for a short survey on these results. The solutions to fractional diffusion equations are strictly related with stable densities. Indeed, the stochastic solutions of time fractional  diffusions  can be realized through time-change by inverse stable subordinators.   A couple of recent works in this field are \cite{mnv-09, OB09}.

%{\bf The Model}
It might come natural to add just  a noise term to the time fractional diffusion and  study the  equation
 \begin{equation}\label{tfspde-0}
 \partial^\beta_tu_t(x)=\Delta u_t(x)+\stackrel{\cdot}{W}(t,x);\ \  u_t(x)|_{t=0}=u_0(x),
 \end{equation}
 where $\stackrel{\cdot}{W}(t,x)$ is a   space-time white noise with $x\in \rd$.

We will make use of {\bf time fractional Duhamel's principle} \cite{umarov-06, umarov-12,Umarov-saydamatov} to get the correct version of \eqref{tfspde-0}.
Let $G_t(x)$ be the fundamental solution of the time fractional PDE $\partial_t^\beta u=\Delta  u$.
 The solution to the time-fractional PDE with force term $f(t,x)$
 \begin{equation}\label{tfpde}
 \partial^\beta_tu_t(x)=\Delta  u_t(x)+f(t,x);\ \  u_t(x)|_{t=0}=u_0(x),
 \end{equation}
 is given by Duhamel's principle, the influence of the external force $f(t,x)$ to the output can be count as
 \begin{equation}
 \partial^\beta_tV(\tau, t,x)=\Delta V(\tau, t,x);\ \  V(\tau,\tau,x)= \partial^{1-\beta}_t f (t, x)|_{t=\tau},
 \end{equation}
  which has solution
 $$
 V(t,\tau, x)=\int_{\rd}G_{t-\tau}(x-y) \partial^{1-\beta}_\tau f (\tau, y)dy.
 $$
Hence solution to \eqref{tfpde} is given by

$$u(t,x)=\int_\rd G_{t}(x-y)u_0(y)dy+\int_0^t\int_\rd G_{t-\tau}(x-y)\partial^{1-\beta}_r f(r,y)dydr.$$

 Hence if we use time fractional Duhamel's principle we will get the mild (integral) solution of
 \eqref{tfspde-0}
   to be of the form (informally):
 \begin{equation}\label{mild-sol-der-noise}\begin{split}
u(t,x)&=\int_\rd G_t(x-y)u_0(y)dy\\
\ \ \ \ &+\int_0^t\int_\rd G_{t-\tau}(x-y) \partial^{1-\beta}_r[ \stackrel{\cdot}{W}(r,y)]dydr.
\end{split}\end{equation}
It is not clear  what the fractional derivative of the space-time white noise mean.

We can remove  the fractional derivative of the noise term  in \eqref{mild-sol-der-noise} in the following way.
Let $\gamma>0$, define the fractional integral by
$$I^{\gamma}_tf(t):=\frac{1}{\Gamma(\gamma)} \int _0^t(t-\tau)^{\gamma-1}f(\tau)d\tau.$$
For every  $\beta\in (0,1)$, and   $g\in L^\infty(\R_+)$ or $g\in C(\R_+)$
$$ \partial _t^\beta I^\beta_t g(t)=g(t).$$
We consider the time fractional PDE with a force given by $f(t,x)=I^{1-\beta}_tg(t,x)$, then by the Duhamel's principle the  mild solution to
\eqref{tfpde} will be given by
$$u_t(x)=\int_{\rd} G_t(x-y)u_0(y)dy+\int_0^t\int_{\rd} G_{t-r}(x-y) g(r,y)dydr.$$

Hence, the preceding discussion suggest  that the correct TFSPDE is the
 following model problem:
 \begin{equation}\label{tfspde-1}
 \partial^\beta_tu_t(x)=\Delta u_t(x)+I^{1-\beta}_t[\stackrel{\cdot}{W}(t,x)];\ \  u_t(x)|_{t=0}=u_0(x).
 \end{equation}
When $ d=1$, the fractional integral above in equation \eqref{tfspde-1}  is defined as
 $$I^{1-\beta}_t[\sigma(u)\stackrel{\cdot}{W}(t,x)]= \frac{1}{\Gamma(1-\beta)} \int _0^t(t-\tau)^{-\beta}\frac{\partial W(d\tau,x)}{\partial x},$$  is well defined only when $0<\beta<1/2$!  %It is a type of % {\color{red}Rieman-Liouville process!}
 By the Duhamel's principle, mentioned above,    mild (integral) solution of  \eqref{tfspde-1} will be  (informally):
\begin{equation}\label{mild-sol-tfspde}\begin{split}
u_t(x)&=\int_\rd G_t(x-y)u_0(y)dy+\int_0^t\int_{\R^d} G_{t-r}(x-y)W(dydr).
\end{split}
\end{equation}

{ Next we want to give a {\bf Physical motivation} to study time fractional SPDEs.}
The time-fractional SPDEs studied in this paper may arise naturally by considering the heat equation in a material with thermal memory; see \cite{chen-kim-kim-2014}  for more details:
Let $u_t(x), e(t,x)$ and $\stackrel{\to}{F}(t,x)$ denote the body temperature, internal energy and flux density, respectively.
Then the relations
\begin{equation}\begin{split}
\partial_te(t,x)&=-div \stackrel{\to}{F}(t,x),\\
e(t,x)=\beta u_t(x), \ \ \stackrel{\to}{F}(t,x)&=-\lambda\nabla u_t(x),
\end{split}\end{equation}
yields the classical heat equation $\beta\partial_tu=\lambda\Delta u$.

According to the law of classical heat equation, the speed of heat flow is infinite
but  the propagation speed can be finite because the heat flow can be disrupted by the response of the material.
In a material with thermal memory Lunardi and Sinestrari \cite{lunardi-sinestrari},  von Wolfersdorf \cite{von-wolfersdorf}  showed that
$$e(t,x)=\bar{\beta}u_t(x)+\int_0^tn(t-s)u_s(x)ds,$$
holds with some appropriate constant $\bar{\beta}$ and kernel $n$.  In most cases we would have $n(t)=\Gamma(1-\beta_1)^{-1}t^{-\beta_1}$.
The convolution implies that the nearer past affects the present more.
If in addition the internal energy also depends on past random effects, then
\begin{equation}\label{phys-1}
\begin{split}
e(t,x)&=\bar{\beta}u_t(x)+\int_0^tn(t-s)u_s(x)ds\\
\ \ \ &+\int_0^t l(t-s)h(s, u_s(x))\frac{\partial {W}(ds, x)}{\partial x},
\end{split}\end{equation}
where $W$ is the space time white noise modeling the random effects.
Take $l(t)=\Gamma(2-\beta_2)^{-1}t^{1-\beta_2}$,
then after differentiation \eqref{phys-1} gives
$$
\partial_t^{\beta_1}u=div\stackrel{\to}{F}+ \frac{1}{\Gamma(1-\beta_2)} \int _0^t(t-s)^{-\beta_2}h(s,u_s(x))\frac{\partial W(ds,x)}{\partial x}.$$

%Let $\gamma>0$, define the fractional integral by
%$$I^{\gamma}_tf(t):=\frac{1}{\Gamma(\gamma)} \int _0^t(t-\tau)^{\gamma-1}f(\tau)d\tau.$$

%For every  $\beta>0$, and   $g\in L^\infty(\rr_+)$ or $g\in C(\rr_+)$
%$$ \partial _t^\beta I^\beta_t g(t)=g(t).$$

 {In this paper we will study existence and uniqueness of  mild solutions to this type of stochastic equations and  its extensions:}
 %\begin{equation}\label{tfspde}
% \partial^\beta_tu(t,x)=L_x u(t,x)+I^{1-\beta}_t[\sigma(u)\stackrel{\cdot}{W}(t,x)];\ \  u(0,x)=u^0(x),
% \end{equation}
 \begin{equation}\label{tfspde}
\begin{split}
 \partial^\beta_tu_t(x)&=-\nu(-\Delta)^{\alpha/2} u_t(x)+I^{1-\beta}_t[\sigma(u)\stackrel{\cdot}{W}(t,x)],\ \ \nu>0, t> 0,\, x\in\R^d;\\
 u_t(x)|_{t=0}&=u_0(x),
 \end{split}
 \end{equation}
where the initial datum $u_0$ is measurable and bounded,
 $-(-\Delta)^{\alpha/2} $ is the fractional Laplacian with $\alpha\in (0,2]$,  $\stackrel{\cdot}{W}(t,x)$ is a  space-time white noise with $x\in \R^d$, and $\sigma:\R\to\R$ is a Lipschitz function.
%When $\sigma(u)=1$, the fractional integral above in equation \eqref{tfspde}  is defined as
 %$$I^{1-\beta}_t[\sigma(u)\stackrel{\cdot}{W}(t,x)]= \frac{1}{\Gamma(1-\beta)} \int _0^t(t-\tau)^{-\beta}\sigma(u(\tau, x))\frac{\partial W(d\tau,x)}{\partial x}$$  is well defined only when $0<\beta<1/2$. It is a type of %Rieman-Liouville process.

Let $G_t(x)$ be the fundamental solution
of  the fractional heat type equation
\begin{equation}\label{Eq:Green0}
\partial^\beta_tG_t(x)=-\nu(-\Delta)^{\alpha/2}G_t(x).
\end{equation}
We know that $G_t(x)$ is the density function of $X(E_t)$, where $X$ is an isotropic $\a$-stable L\'evy process in $\R^d$ and $E_t$ is the first passage time of a $\beta$-stable subordinator $D=\{D_r,\,r\ge0\}$, or the inverse stable subordinator of index $\beta$: see, for example, Bertoin \cite{bertoin} for properties of these processes, Baeumer and Meerschaert \cite{fracCauchy} for more on time fractional diffusion equations, and Meerschaert and Scheffler  \cite{limitCTRW} for properties of the inverse stable subordinator $E_t$.

Let $p_{{X(s)}}(x)$ and $f_{E_t}(s)$ be the density of $X(s)$ and $E_t$, respectively. Then the Fourier transform of $p_{{X(s)}}(x)$ is given by
\begin{equation}\label{Eq:F_pX}
\widehat{p_{X(s)}}(\xi)=e^{-s\nu|\xi|^\a},
\end{equation}
and
\begin{equation}\label{Etdens0}
f_{E_t}(x)=t\beta^{-1}x^{-1-1/\beta}g_\beta(tx^{-1/\beta}),
\end{equation}
where $g_\beta(\cdot)$ is the density function of $D_1.$ The function $g_\beta(u)$ (cf. Meerschaert and Straka \cite{meerschaert-straka}) is infinitely differentiable on the entire real line, with $g_\beta(u)=0$ for $u\le 0$.

%\begin{equation}\label{Eq:gbeta0}
%g_\beta(u)\sim K(\beta/u)^{(1-\beta/2)/(1-\beta)}\exp\{-|1-\beta|(u/\beta)^{\beta/(\beta-1)}\}\quad\mbox{as}\,\, u\to0+,
%\end{equation}
%and
%\begin{equation}\label{Eq:gbetainf}
%g_\beta(u)\sim\frac{\beta}{\Gamma(1-\beta)}u^{-\beta-1} \quad\mbox{as}\,\, u\to\infty.
%\end{equation}
By conditioning, we have
\begin{equation}\label{Eq:Green1}
G_t(x)=\int_{0}^\infty p_{_{X(s)}}(x) f_{E_t}(s)ds.
\end{equation}

 A related  time-fractional  SPDE was studied by  Karczewska \cite{karczewska}, Chen  et al. \cite{chen-kim-kim-2014}, and Baeumer et al \cite{baeumer-Geissert-Kovacs}. They have proved regularity of the solutions to the time-fractional parabolic type SPDEs using cylindrical Brownain motion in Banach spaces  in line with the methods in \cite{daPrato-Zabczyk}.

In this paper we study the existence and uniqueness of the solution to (\ref{tfspde}) under global Lipchitz conditions on $\sigma$, using the white noise  approach of \cite{walsh}:
%Assume that $\sigma(\cdot)$ satisfies the following global Lipschitz condition, i.e. there exists a generic positive constant $A$ such that and growth conditions:
%\begin{equation}\label{Eq:Cond_a}
%|\sigma(x)-\sigma(y)|\le A|x-y|\quad\mbox{for all }\,\,x,\,y\in\R.
%\end{equation}
%Clearly, (\ref{Eq:Cond_a}) implies the uniform linear growth condition of $\sigma(\cdot).$
We say that
an $\mathcal{F}_t$-adapted random field $\{u_t(x),\,t\ge 0,\,x\in\R^d\}$ is a mild solution of (\ref{tfspde}) with initial value $u_0$ if the following integral equation is fulfilled
\begin{equation}\label{Eq:Mild}
u_t(x)=\int_{\R^d}u_0(y)G_t(x-y)dy+\int_0^t\int_{\R^d}\sigma(u_r(y))G_{t-r}(x-y)W(drdy).
\end{equation}

For a comparison of the two approaches to SPDE's see the paper by Dalang and Quer-Sardanyons \cite{Dalang-Quer-Sardanyons}.

%Let $T$ be a fixed positive number, and let $B_{T,p}$ denote the family of all $\mathcal{F}_t$-adapted random fields $\{u(t,x),\,t\in [0,\,T],\,x\in\R^d\}$ satisfying
%\begin{equation}\label{Eq:SquareInt}
%\sup_{x\in\R^d}\sup_{t\in[0,\,T]}\E\left[|u(t,x)|^p\right]<\infty,
%\end{equation}
%with the convention that $B_{T,2}=B_T$. It is easy to check that for each fixed $T$ and $p,$ $B_{T,p}$ is a Banach space.

%\begin{tm}[
%Nane et al. \cite{nane-et-al-2014} proved the existence and uniqueness result  for the equation \eqref{tfspde} when
%]\label{Thm:ExUnq}
% $d<\min\{2,\beta^{-1}\}\a$,  equation (\ref{tfspde}) subject to (\ref{Eq:Initial}) and global Lipschitz condition on $\sigma$ has an a.s.-unique solution %$u(t,x)$ that satisfies that for all $T>0,$
%$u(t,x)\in B_{T,p}.$
%\end{tm}

%\section{The model problem}\label{model}
%{\color{red} add some intro about the intermittency here}
%{\color{red} add a paragraph that tells about the sections and the place of the main results!}

%{\color{red}Please look at how I described the paper organization: It would be better if section 6 results about existence and uniqueness of solution comes before section 4}

{We now briefly give an outline of the paper. We adapt the methods of  proofs of the results in  \cite{khoshnevisan-cbms} with  many crucial nontrivial changes. We give some preliminary results in section 2. Moment estimates for time increments and spatial increments of the solution is given in Section 3. The main result in this section is Theorem \ref{thm:continuity-stoch-convolution}. In section 4, we give main results of the paper about existence and uniqueness, and the  continuity of the solution to the time fractional SPDEs. The main result is   Theorem \ref{thm:existence-moment}. In Section 5, we show that under faster than linear growth of $\sigma$ there is  no finite energy solution. In section 6, we discuss the  equivalence of time-fractional SPDEs to higher order SPDEs.
Throughout the paper, we use the letter $C$ or $c$ with or without subscripts to denote
a constant whose value is not important and may vary from places to places.}
%Finally, The existence and uniqueness of solution of  \eqref{Eq:modelspde} will be proved in section 6.

\section{Preliminaries}

In this section, we give some preliminary results that will be needed in the remaining sections of the paper.

Applying the Laplace transform with respect to time variable t, Fourier transform with respect to space variable x.
 Laplace-Fourier transform of $G$ defined in  \eqref{Eq:Green1} is given by
\begin{eqnarray}
\int_{0}^{\infty}\int_{{\R^d}} e^{-\lambda t + i\xi\cdot x}G_t(x)dxdt&=&\int_{0}^{\infty}e^{-\lambda t}dt\int_{0}^\infty f_{E_t}(s)ds\int_{{\R^d}} e^{i\xi\cdot x}p_{X(s)}(x)dx\nonumber\\
&=&\int_{0}^{\infty}e^{-s\nu|\xi|^\alpha}ds\int_{0}^\infty e^{-\lambda t}f_{E_t}(s)dt\nonumber\\
&=&\frac{\lambda^\beta}{\lambda}\int_{0}^\infty e^{-s(\nu|\xi|^\alpha +\lambda^\beta) }ds\nonumber\\
&=&\frac{\lambda^{\beta-1}}{\lambda^{\beta} + \nu|\xi|^\alpha},
\end{eqnarray}
here we used the fact that the laplace transform $t\to\lambda $ of $f_{E_t}(u)$ is given by $\lambda^{\beta-1}e^{-u\lambda^\beta}$.
Using the convention, $\sim$ to denote the Laplace transform and $\ast$ the Fourier transform we get
\begin{equation}
\tilde{G}^\ast_t(x) = \frac{\lambda^{\beta-1}}{\lambda^{\beta} +\nu |\xi|^\alpha}.
\end{equation}
Inverting the Laplace transform, it yields
\begin{equation}\label{fouriertransformofG}
G^\ast_t(\xi) = E_\beta(-\nu|\xi|^\alpha t^\beta),
\end{equation}
where
\begin{equation}\label{ML-function}
E_\beta(x) = \sum_{k=0}^\infty\frac{x^k}{\Gamma(1+\beta k)},
\end{equation}
 is the Mittag-Leffler function. In order to invert the Fourier transform, we will make use of the integral \cite[eq. 12.9]{haubold-mathai-saxena}
\begin{equation}
\int_0^\infty\cos(ks)E_{\beta,\a}(-as^\mu)ds = \frac{\pi}{k}H_{3,3}^{2,1}\bigg[\frac{k^\mu}{a}\bigg|^{(1,1), (\a,\beta), (1,\mu/2)}_{(1,\mu),(1,1),(1,\mu/2)}\bigg],\nonumber
\end{equation}
%{\color{red} Could you explain this notation of H-function?}{\color{blue}Done! I gave A reference since it is long to define the H-function}
where $\mathcal{R}(\a)>0,\mathcal{\beta}>0,k>0,a>0, H_{p,q}^{m,n}$ is the H-function given in \cite[Definition 1.9.1, p. 55]{mathai} and the formula
\begin{equation}
\frac{1}{2\pi}\int_{-\infty}^\infty e^{-i\xi x}f(\xi)d\xi = \frac{1}{\pi}\int_0^\infty f(\xi)\cos(\xi x)d\xi.\nonumber
\end{equation}
Then this gives the function as
\begin{equation}\label{G-function}
G_t(x) = \frac{1}{|x|} H_{3,3}^{2,1}\bigg[\frac{|x|^\a}{\nu t^\beta}\bigg|^{(1,1), (1,\beta), (1,\a/2)}_{(1,\a),(1,1),(1,\a/2)}\bigg].
\end{equation}
Note that for $\a = 2$ using reduction formula for the H-function we have
\begin{equation}
G_t(x) = \frac{1}{|x|}H^{1,0}_{1,1}\bigg[\frac{|x|^2}{\nu t^\beta}\bigg|^{(1,\beta)}_{(1,2)}\bigg]
\end{equation}
Note  that for $\beta = 1$ it reduces to the Gaussian density
\begin{equation}
G_t(x) = \frac{1}{(4\nu\pi t)^{1/2}}\exp\left(-\frac{|x|^2}{4\nu t}\right).
\end{equation}
Recall
 uniform estimate of Mittag-Leffler function \cite[Theorem 4]{simon}
 \begin{equation}\label{uniformbound}
 \frac{1}{1 + \Gamma(1-\beta)x}\leq E_{\beta}(-x)\leq \frac{1}{1+\Gamma(1+\beta)^{-1}x} \ \  \ \text{for}\ x>0.
 \end{equation}

\begin{lem}\label{Lem:Green1} For $d < 2\a,$
\begin{equation}\label{Eq:Greenint}
\int_{{\R^d}}G^2_t(x)dx  =C^\ast t^{-\beta d/\a}
\end{equation}
where $C^\ast = \frac{(\nu )^{-d/\a}2\pi^{d/2}}{\a\Gamma(\frac d2)}\frac{1}{(2\pi)^d}\int_0^\infty z^{d/\a-1} (E_\beta(-z))^2 dz.$
\end{lem}

\begin{proof}
Using Plancherel theorem and \eqref{fouriertransformofG}, we have%, \eqref{uniformbound}, and Mellin transform of Mittag-Leffler function we have
%\begin{eqnarray}
%\int_{-\infty}^{\infty}|G(t,x)|^2 dx &=& \frac{1}{2\pi}\int_{-\infty}^{\infty}|G^\ast(t,\xi)|^2 d\xi\\
%&=& \frac{1}{2\pi}\int_{-\infty}^{\infty}|E_\beta(-|\xi|^\a t^\beta)|^2 d\xi\nonumber\\
%&=& \frac{1}{\pi}\int_{0}^{\infty}E_\beta(-\xi^\a (2t)^\beta) d\xi\nonumber\\
%&=& \frac{1}{\pi \a}\int_{0}^{\infty}u^{1/\a - 1}E_\beta(-(2t)^\beta u) du\nonumber\\
%&=&\frac{1}{\pi\a 2^{\beta/\a}}\frac{\Gamma(1/\a)\Gamma(1-1/\a)}{\Gamma(1-\beta/\a)} t^{-\beta/\a }\nonumber.
%\end{eqnarray}

%\begin{rk}
%\begin{eqnarray}
%\int_{\rd}|G_t(x)|^2 dx &=& \frac{1}{(2\pi)^d}\int_{\rd}|G^\ast_t(\xi)|^2 d\xi\\
%&=& \frac{1}{(2\pi)^d}\int_{\rd}|E_\beta(-\nu|\xi|^\a t^\beta)|^2 d\xi\nonumber\\
%&=& \frac{1}{(2\pi)^d}\int_{\rd}E_\beta(-\nu|\xi|^\a (2t)^\beta) d\xi\nonumber\\
%&=& \frac{2\pi^{d/2}}{\Gamma(\frac d2)}\frac{1}{(2\pi)^d}\int_0^\infty r^{d-1} E_\beta(-\nu r^\a (2t)^\beta) dr\nonumber\\
%&=&\frac{2\pi^{d/2}}{\Gamma(\frac d2)}\frac{1}{\a (2\pi)^d}\int_{0}^{\infty}u^{d/\a - 1}E_\beta(-\nu (2t)^\beta u) du\nonumber\\
%&=&\frac{2^{1-\beta d/\a}\pi^{d/2}}{\Gamma(\frac d2)}\frac{1}{\a (2\pi)^d}\frac{\Gamma(d/\a)\Gamma(1-d/\a)}{\Gamma(1-\beta d/\a)}\nu^{-d/\a}t^{-\beta d/\a}\nonumber.
%\end{eqnarray}
\begin{eqnarray}
\int_{\rd}|G_t(x)|^2 dx &=& \frac{1}{(2\pi)^d}\int_{\rd}|G^\ast_t(\xi)|^2 d\xi= \frac{1}{(2\pi)^d}\int_{\rd}|E_\beta(-\nu|\xi|^\a t^\beta)|^2 d\xi\nonumber\\
&=&\frac{2\pi^{d/2}}{\Gamma(\frac d2)}\frac{1}{(2\pi)^d}\int_0^\infty r^{d-1} (E_\beta(-\nu r^\a t^\beta))^2 dr.\label{square-kernel}\\
&=&\frac{(\nu t^\beta)^{-d/\a}2\pi^{d/2}}{\a\Gamma(\frac d2)}\frac{1}{(2\pi)^d}\int_0^\infty z^{d/\a-1} (E_\beta(-z))^2 dz.
\end{eqnarray}
We used the integration in  polar coordinates for radially symmetric function in  the last equation above. Now using equation \eqref{uniformbound} we get
\begin{eqnarray}\label{uniformboundforE}
\int_{0}^\infty\frac{z^{d/\a-1} }{(1+\Gamma(1-\beta)z)^2} dr&\leq&\int_0^\infty z^{d/\a-1} (E_\beta(-z))^2 dz\nonumber\\
&\leq&\int_{0}^\infty\frac{z^{d/\a-1} }{(1+\Gamma(1+\beta)^{-1}z)^2}dz
\end{eqnarray}
Hence  $\int_0^\infty z^{d/\a-1} (E_\beta(-z))^2 dz<\infty$ if and only if $d<2\a$. In this case
%The lower bound in equation \eqref{square-kernel} is
%begin{eqnarray}\int_{0}^\infty\frac{r^{d-1}}{(1+\Gamma(1-\beta)\nu r^\a t^\beta)^2} dr &=& \frac{1}{\a\Gamma(1-\beta)^{d/\a}(\nu %t^\beta)^{d/\a}}\int_0^\infty\frac{u^{d/\a -1}}{(1+u)^2}du\nonumber\\
%&=& \frac{\text{B}(d/\a, 2-d/\a)}{\a\Gamma(1-\beta)^{d/\a}(\nu t^\beta)^{d/\a}},\nonumber
%\end{eqnarray}
 the upper bound in equation \eqref{uniformboundforE} is
$$\int_0^\infty \frac{z^{d/\a-1} }{(1+\Gamma(1+\beta)^{-1}z)^2} dz = \frac{\text{B}(d/\a, 2-d/\a)}{\Gamma(1+\beta)^{-d/\a}},$$
where $\text{B}(d/\a, 2-d/\a)$ is a Beta function.
\end{proof}
\begin{rk}For $d<2\a$,
\begin{equation}\label{Eq:Greenint2}
\frac{\text{B}(d/\a, 2-d/\a)}{\Gamma(1-\beta)^{d/\a}} z^{-\beta d/\a}\leq\int_{0}^\infty z^{d/\a -1}(E_\beta(-z))^2dz \leq \frac{\text{B}(d/\a, 2-d/\a)}{\Gamma(1+\beta)^{-d/\a}} z^{-\beta d/\a}.\nonumber
\end{equation}
%where $C_0 = \frac{\Gamma(d/\a)\Gamma( 2-d/\a)}{\a\Gamma(d/2)2^{d-1}\pi^{d/2}}\nu^{-d/\a}.$
\end{rk}
%\begin{rk}
%For  special case $d = 1, \a = 2$ and $\beta = 1$ we get
%\begin{equation}\label{specialcase}
%\int_{-\infty}^{\infty}G_t(x)^2 dx = \frac{1}{(8\nu\pi t)^{1/2}}.
%\end{equation}
%\end{rk}
%{\color{red} Check if the following lemma is ok. We have it in $\rd$ now!}{\color{blue} It is correct since we can split it into n integrals}
\begin{lem} For $\lambda\in\rd$ and $\a = 2$,
$$\int_{\rd}e^{\lambda\cdot x}G_s(x)dx = E_\beta(\nu |\lambda|^2 s^\beta).$$
\end{lem}
\begin{proof}
Using uniqueness of Laplace transform we can easily show $\E[E_s^k] = \frac{\Gamma(1+k)s^{\beta k}}{\Gamma(1+\beta k)}$ for $k>-1$. Therefore,
using this and moment-generating function of Gaussian densities we have
\begin{eqnarray}
\int_{\rd}e^{\lambda\cdot x}G_s(x)dx &=& \int_{0}^\infty \int_{\rd}e^{\lambda\cdot x}\frac{e^{-\frac{|x|^2}{4\nu u}}}{(4\pi\nu u)^{d/2}}dxf_{E_s}(u)du\nonumber\\
&=& \int_{0}^\infty e^{\nu |\lambda|^2 u}f_{E_s}(u)du\nonumber\\
&=&\sum_{k=0}^{\infty}\frac{\nu^k|\lambda|^{2k}}{k!}\int_0^\infty u^kf_{E_s}(u)du\nonumber\\
&=& \sum_{k=0}^{\infty}\frac{\nu^k|\lambda|^{2k}}{k!}\frac{\Gamma(1+k)s^{\beta k}}{\Gamma(1+\beta k)}.
\end{eqnarray}

\end{proof}

$\stackrel{\cdot}{W}(t,x)$ is a space-time white noise  with $x\in \R^d$, which is assumed to be adapted with respect to a filtered probability space $(\Omega, \mathcal{F},  \mathcal{F}_t, \P)$, where $ \mathcal{F}$ is complete and the filtration $\{\mathcal{F}_t, t\geq 0\}$ is right continuous. $\stackrel{\cdot}{W}(t,x)$ is  generalized processes with covariance given by
$$
\E\bigg[\stackrel{\cdot}{W}(t,x) \stackrel{\cdot}{W}(s,y)\bigg]=\delta(x-y)\delta(t-s).
$$
That is, $W(f) $ is a random field indexed by functions $ f\in L^2((0,\infty)\times \R^d )$ and for all $f,g\in L^2((0,\infty)\times \R ^d)$, we have
$$
\E\bigg[W(f)W(g)\bigg]=\int_0^\infty \int_{\R^d} f(t,x)g(t,x)dxdt.
$$

Hence $W(f)$ can be represented as
$$
W(f)=\int_0^\infty \int_{\R^d} f(t,x)W(dxdt).
$$
Note that $W(f)$ is $\mathcal{F}_t$-measurable whenever $f$ is supported on $[0,t]\times\R^d$.

{ Next we give  the  definition of Walsh-Dalang Integrals that is used in equation \eqref{Eq:Mild}.}
We use the Brownian Filtration $\{\mathcal{F}_t\}$ and the Walsh-Dalang integrals defined as follows:
\begin{itemize}
\item $(t,x)\to \Phi_t(x)$ is an elementary random field when $\exists 0\leq a<b$ and an $\mathcal{F}_a$-measurable $X\in L^2(\Omega)$ and $\phi\in L^2({\rd})$ such that
    $$
    \Phi_t(x)=X1_{[a,b]}(t)\phi(x)\ \ \ (t>0, x\in \rd).
    $$
%\item A random field $\Phi$ is simple  if $\exists$ elementary random fields $\Phi^{(1)}, \cdots, \Phi^{(n)}$, with disjoint supports, such that
 %       $\Phi=\sum_{i=1}^n\Phi^{(i)}$.
 \item  If $h=h_t(x)$ is non-random and $\Phi$ is elementary, then
 $$
 \int h\Phi dW:=X\int_{(a,b)\times\rd}h_t(x)\phi(x)W(dtdx).
 $$
 \item The stochastic integral is Wiener's, and it is  well defined iff $h_t(x)\phi(x)\in L^2([a,b]\times\rd)$.
 \item We have Walsh isometry,
 $$
 \E \bigg(\bigg|\int h\Phi dW\bigg|^2\bigg)=\int_0^\infty\int_{\rd}dy[h_s(y)]^2\E(|\Phi_s(y)|^2).
 $$

\end{itemize}

Let $\Phi$ be a random field, and for every $\gamma>0$ and $k\in [2, \infty)$ define
\begin{equation}
\mathcal{N}_{\gamma,k}(\Phi) := \sup_{t\geq 0}\sup_{x\in\R^d}\left(e^{-\gamma t}||\Phi_t(x)||_k\right)=\sup_{t\geq 0}\sup_{x\in\R^d}\left(e^{-\gamma t}\bigg(\E[|\Phi_t(x)|^k]\bigg)^{1/k}\right).
\end{equation}
%The special case $k=1$ has already played a role in our construction of stochastic integrals.
If we identify a.s.-equal random fields, then every $\mathcal{N}_{\gamma, k} $ becomes a norm. Moreover, $\mathcal{N}_{\gamma, k} $ and $\mathcal{N}_{\gamma', k} $ are equivalent norms for all $\gamma, \gamma'>0$ and $k\in [2,\infty).$ Finally, we note that if $\mathcal{N}_{\gamma, k}(\Phi) <\infty$ for some $\gamma >0$ and $k\in(2,\infty)$, then $\mathcal{N}_{\gamma, 2}(\Phi)<\infty$  as well, thanks to Jensen's inequality.

\begin{defin}\label{comp-space}
We denote by $\mathcal{L}^{\gamma, 2}$  the completion of the space of all simple  random fields in the norm $\mathcal{N}_{\gamma, 2}.$
\end{defin}

Given a random field $\Phi := \{\Phi_t(x)\}_{t\geq 0, x\in\R^d}$ and space-time noise $W$, we define the [space-time] \textit{stochastic convolution} $G\circledast \Phi$ to be the random field that is defined as
$$(G\circledast \Phi)_t(x) := \int_{(0,t)\times \R^d} G_{t-s}(y-x)\Phi_s(y)W(dsdy),$$
for $t>0$ and $x\in\R^d,$ and $(G\circledast W)_0(x) := 0.$

We can understand the properties of $G\circledast \Phi$ for every fixed $t>0$ and $x\in\R^d$ as follows. Define
\begin{equation}
G_s^{(t,x)}(y) := G_{t-s}(y-x)\cdot {\bf{1}}_{(0,t)}(s)\ \ \ \mbox{for all}\ s\geq 0\ \mbox{and}\ y\in\R^d.
\end{equation}
Clearly, $G^{(t,x)}\in L^2(\R_+\times \R^d)$ for $d<\min\{2,\beta^{-1}\}\a$; in fact,
$$\int_0^\infty ds \int_{\R^d} [G_s^{(t,x)}(y)]^2dy =  \int_0^t ds \int_{\R^d} [G_{s}(y)]^2dy = C^\ast t^{1-\beta d/\a}<\infty.$$
This computation follows from Lemma \ref{Lem:Green1}. Thus, we may interpret the random variable $(G\circledast \Phi)_t(x)$ as the stochastic integral $\int G_s^{(t,x)}\Phi dW$, provided that $\Phi$ is in $\mathcal{L}^{\beta, 2}$ for some $\beta>0.$ Let us recall that $\Phi\mapsto G\circledast \Phi$ is a random linear map; that is, if $\Phi, \Psi
\in \mathcal{L}^{\beta, 2}$ for some $\beta > 0,$ then for all $a, b\in\R$ the following holds almost surely:
\begin{eqnarray}
&&\int_{(0,t)\times \R^d} G_{t-s}(y-x)(a\Phi_s(y) + b\Psi(y))W(dsdy)\nonumber\\
&=&a\int_{(0,t)\times \R^d} G_{t-s}(y-x)\Phi_s(y)W(dsdy) + b\int_{(0,t)\times \R^d} G_{t-s}(y-x)\Psi_s(y)W(dsdy)\nonumber
\end{eqnarray}

\section{Estimates on the moments of the increments of the solution}
In this section we prove continuity of various increments related to the solution of \eqref{tfspde}. We start with the next result
\begin{lem}\label{DiffEst}
Suppose $d<\min(2,\beta^{-1})\a$, and  $p\ge 2$, then we have the following estimates for time increments of the density G.

\noindent (i). For $t\leq t',$ %and $\beta d/\a<1$,
 we have
\begin{eqnarray}%\frac{(2^{\beta d/\a}-2)C^\ast}{(1-\beta d/\a)}(t' - t)^{1-\beta d/\a}&\leq&
\int_0^t\int_{\R^d}[G_{t'-r}(x-y)-G_{t-r}(x-y)]^2drdy
\leq\frac{C^\ast}{(1-\beta d/\a)}(t' - t)^{1-\beta d/\a},\nonumber
\end{eqnarray}
where $C^\ast$ is a constant given in Lemma \ref{Lem:Green1}.

\noindent (ii). For  $x,x'\in\R^d,$ we have
\begin{eqnarray}c|x'-x|^2&\leq&\int_0^t\int_{\R^d}[G_{t-r}(x-y)-G_{t-r}(x'-y)]^2drdy\nonumber\\
&\leq& C|x'-x|^{\min\left\{\left(\frac{\a-\beta d}{\beta}\right)^{-},2\right\}}.
\end{eqnarray}
\end{lem}

\begin{proof}
Using $G^\ast_{t'-r}(x-\cdot)- G^\ast_{t-r}(x-\cdot) = e^{i\xi\cdot x}G^\ast_{t'-r} (\xi)-e^{i\xi\cdot x}G^\ast_{t-r}(\xi)$, Plancherel theorem and computation in Lemma \ref{Lem:Green1} we have
\begin{eqnarray}
&&\int_{\R^d}|G_{t'-r}(x-y)-G_{t-r}(x-y)|^2dy\nonumber\\
 &=&\frac{1}{(2\pi)^d}\int_{\R^d}|G^\ast_{t'-r}(\xi)- G^\ast_{t-r}(\xi)|^2 d\xi\nonumber\\
 &=&\frac{1}{(2\pi)^d}\int_{\R^d} (G^\ast_{t'-r}(\xi))^2d\xi + \frac{1}{(2\pi)^d}\int_{\R^d} (G^\ast_{t-r}(\xi))^2d\xi \nonumber\\
 &-&\frac{2}{(2\pi)^d}\int_{\R^d} G^\ast_{t'-r}(\xi)G^\ast_{t-r}(\xi)d\xi \nonumber\\
 &=& C^\ast (t'- r)^{-\beta d/\a} + C^\ast (t-r)^{-\beta d/\a} -\frac{2}{(2\pi)^d}\int_{\R^d} G^\ast_{t'-r}(\xi)G^\ast_{t-r}(\xi)d\xi.\nonumber
\end{eqnarray}

Using  Plancharel theorem, integration in polar coordinates in $\rd$, and the fact that $f(z)=E_\beta(-z)$ is decreasing (since it is completely monotonic, i.e. $(-1)^nf^{(n)}(z)\geq 0$ for all $z>0$, $n=0,1,2,3,\cdots$), we get
\begin{eqnarray}
 &&\frac{2}{(2\pi)^d}\int_{\R^d} G^\ast_{t'-r}(\xi)G^\ast_{t-r}(\xi)d\xi \nonumber\\
 &=&\frac{2}{(2\pi)^d}\int_{\R^d}E_\beta(-\nu(t'-r)^\beta |\xi|^\alpha)E_\beta(-\nu(t-r)^\beta |\xi|^\alpha)d\xi \nonumber\\
 &=&\frac{2\pi^{d/2}}{\Gamma(\frac d2)}\frac{2}{(2\pi)^d}\int_{0}^\infty  s^{d-1}E_\beta(-\nu(t'-r)^\beta s^\alpha)E_\beta(-\nu(t-r)^\beta s^\alpha)ds \nonumber\\
 &\geq &\frac{2\pi^{d/2}}{\Gamma(\frac d2)}\frac{2}{(2\pi)^d}\int_{0}^\infty s^{d-1}E_\beta(-\nu(t'-r)^\beta s^\alpha)E_\beta(-\nu(t'-r)^\beta s^\alpha)ds \nonumber\\
&\geq &\frac{2\pi^{d/2}}{\Gamma(\frac d2)}\frac{2}{(2\pi)^d}\int_{0}^\infty s^{d-1}\bigg( E_\beta(-\nu(t'-r)^\beta s^\alpha)\bigg)^2ds \nonumber\\
 &=& \frac{2\pi^{d/2}}{\Gamma(\frac d2)}\frac{2 \nu^{- d/\a}(t'-r)^{-\beta d/\a}}{\a (2\pi)^d}\int_{0}^\infty z^{d/\a-1}\bigg( E_\beta(-z)\bigg)^2dz \nonumber\\
 &=& 2C^*(t'-r)^{-\beta d/\a}.\nonumber
\end{eqnarray}
%here $C(volume)=\frac{2\pi^{d/2}}{\Gamma(\frac d2)}$ is the volume of the unit sphere in $\rd$.

Now integrating both sides wrt to $r$ from $0$ to $t$ we get
\begin{eqnarray}\label{asmtotitc}
&&\int_0^t\int_{{\R^d}}[G_{t'-r}(x-y)-G_{t-r}(x-y)]^2drdy\\
&\leq &\frac{-C^\ast (t'- t)^{1-\beta d/\a}}{1-\beta d/\a} + \frac{C^\ast (t')^{1-\beta d/\a}}{1-\beta d/\a} + \frac{C^\ast t^{1-\beta d/\a}}{1-\beta d/\a}\nonumber\\
&+&\frac{ 2C^\ast (t'- t)^{1-\beta d/\a}}{1-\beta d/\a}-\frac{ 2C^\ast(t')^{1-\beta d/\a}}{1-\beta d/\a}\nonumber\\
&=&\frac{C^\ast(t'- t)^{1-\beta d/\a}}{1-\beta d/\a}-\frac{C^\ast (t')^{1-\beta d/\a}}{1-\beta d/\a}+\frac{C^\ast t^{1-\beta d/\a}}{1-\beta d/\a}\nonumber\\
&\leq &\frac{C^\ast(t'- t)^{1-\beta d/\a}}{1-\beta d/\a},\nonumber
\end{eqnarray}
the last inequality follows since $t<t'$.\\
%Using the fact that $b^k - a^k\leq (b-a)^k$ for $a\leq b$ and $0<k<1$ we estimate upper bound for \eqref{asmtotitc}. %First we estimate the lower bound
%\begin{eqnarray}
%&&\int_0^t\int_{{\R^d}}[G(t'-r,x-y)-G(t-r,x-y)]^2drdy\nonumber\\
%&\geq&\frac{(2^{\beta d/\a}-1)C^\ast(t'- t)^{1-\beta d/\a}}{1-\beta d/\a}+\frac{C^\ast t^{1-\beta d/\a}}{1-\beta d/\a}-\frac{C^\ast (t')^{1-\beta d/\a}}{1-\beta d/\a}\nonumber\\
%&=&\frac{(2^{\beta d/\a}-1)C^\ast(t'- t)^{1-\beta d/\a}}{1-\beta d/\a}-\frac{C^\ast \left((t')^{1-\beta d/\a}-t^{1-\beta d/\a}\right)}{1-\beta d/\a}\nonumber\\
%&\geq&\frac{(2^{\beta d/\a}-2)C^\ast(t'- t)^{1-\beta d/\a}}{1-\beta d/\a}.\nonumber
%\end{eqnarray}
%Similarly, we estimate the upper bound
%\begin{eqnarray}
%&&\int_0^t\int_{{\R^d}}[G_{t'-r}(x-y)-G_{t-r}(x-y)]^2drdy\nonumber\\
%&\leq&\frac{(2^{\beta d/\a}-1)C^\ast(t'- t)^{1-\beta d/\a}}{1-\beta d/\a}+\frac{C^\ast (t')^{1-\beta d/\a}}{1-\beta d/\a}-\frac{C^\ast t^{1-\beta d/\a}}{1-\beta d/\a}\nonumber\\
%&=&\frac{(2^{\beta d/\a}-1)C^\ast(t'- t)^{1-\beta d/\a}}{1-\beta d/\a}+\frac{C^\ast \left((t')^{1-\beta d/\a}-t^{1-\beta d/\a}\right)}{1-\beta d/\a}\nonumber\\
%&\leq&\frac{2^{\beta d/\a}C^\ast(t'- t)^{1-\beta d/\a}}{1-\beta d/\a}.\nonumber
%\end{eqnarray}
Now we prove (ii). Using $G^\ast_{t-r}(x'-\cdot)- G^\ast_{t-r}(x-\cdot) = e^{i\xi\cdot x'}G^\ast_{t'-r}(\xi)-e^{i\xi\cdot x}G^\ast_{t-r}(\xi)$ and Plancherel theorem, we have
\begin{eqnarray}\label{Eq:GIncre1}
&&\int_0^t\int_{{\R^d}}[G_{t-r}(x'-y)-G_{t-r}(x-y)]^2dydr\nonumber\\
&=&\frac{1}{(2\pi)^d}\int_0^t\int_{{\R^d}}[G^\ast_{t-r}(\xi)]^2\bigg|1 - e^{i\xi(x'-x)}\bigg|^2d\xi dr
\end{eqnarray}
In the following, we divide our proof into two cases:

\noindent Case 1. If $d+2<\min(2,\beta^{-1})\a$, then we apply the following inequality
\begin{equation}\label{Eq:SIncrU3}
\left|1-e^{i\l\xi,x'-x\r}\right|^2=4\sin^2\left(\frac{\l\xi,x'-x\r}{2}\right)\le |\xi|^2|x'-x|^2
\end{equation}
to (\ref{Eq:GIncre1}) to derive that
\begin{eqnarray}
&&\int_0^t\int_{{\R^d}}[G_{t-r}(x'-y)-G_{t-r}(x-y)]^2drdy\nonumber\\
&\leq&\frac{1}{(2\pi)^d}\int_0^t\int_{{\R^d}}[G^\ast_{t-r}(\xi)]^2|\xi|^2|x'-x|^2 d\xi dr\nonumber\\
&=&\frac{|x'-x|^2 2\pi^{d/2}}{(2\pi)^d \Gamma(d/2)}\int_{0}^t\int_0^\infty u^{d+1}(E_\beta(-\nu(t-r)^\beta u^\a))^2dudr\nonumber\\
&\leq&\frac{|x'-x|^22\pi^{d/2}}{(2\pi)^{d} \Gamma(d/2)}\int_0^t\int_0^\infty\frac{u^{d+1}}{(1+\Gamma(1+\beta)^{-1}v(t-r)^{\beta}u^\a)^2}dzdr\nonumber\\
&=&\frac{2\pi^{d/2}\text{B}((d+2)/\a, 2-(d+2)/\a)\nu^{-(d+2)/\a}}{(2\pi)^d\Gamma(d/2)\a(1-(d+2)\beta/\a)\Gamma(1+\beta)^{-(d+2)/\a}}t^{1-(d+2)\beta/\a}|x'-x|^2.\nonumber
\end{eqnarray}
\noindent Case 2. If $d+2\geq\min(2,\beta^{-1})\a$, then we choose $\eps>0$ fixed such that $d+2\eps<\min(2,\beta^{-1})\a$. [This can be done since $d<\min(2,\beta^{-1})\a$.] Clearly, $\eps<1.$
then we apply the following inequality
\begin{equation}\label{Eq:SIncrU5}
\left|1-e^{i\l\xi,x'-x\r}\right|^2\le 4^{1-\eps}|\xi|^{2\eps}|x'-x|^{2\eps}
\end{equation}
to (\ref{Eq:GIncre1}) to derive that
\begin{eqnarray}
&&\int_0^t\int_{{\R^d}}[G_{t-r}(x'-y)-G_{t-r}(x-y)]^2drdy\nonumber\\
&\leq&\frac{4^{1-\eps}}{(2\pi)^d}\int_0^t\int_{{\R^d}}[G^\ast_{t-r}(\xi)]^2|\xi|^{2\eps}|x'-x|^{2\eps} d\xi dr\nonumber\\
&=&\frac{|x'-x|^{2\eps}4^{1-\eps} 2\pi^{d/2}}{(2\pi)^d \Gamma(d/2)}\int_{0}^t\int_0^\infty u^{d+2\eps-1}(E_\beta(-\nu(t-r)^\beta u^\a))^2dudr\nonumber\\
&\leq&\frac{|x'-x|^{2\eps}4^{1-\eps}2\pi^{d/2}}{(2\pi)^{d} \Gamma(d/2)}\int_0^t\int_0^\infty\frac{u^{d+2\eps-1}}{(1+\Gamma(1+\beta)^{-1}v(t-r)^{\beta}u^\a)^2}dzdr\nonumber\\
&=&\frac{2\pi^{d/2}4^{1-\eps}\text{B}((d+2\eps)/\a, 2-(d+2\eps)/\a)\nu^{-(d+2\eps)/\a}}{(2\pi)^d\Gamma(d/2)\a(1-(d+2\eps)\beta/\a)\Gamma(1+\beta)^{-(d+2\eps)/\a}}t^{1-(d+2\eps)\beta/\a}|x'-x|^{2\eps}.\nonumber
\end{eqnarray}

To estimate lower bound we need to use $\frac{1}{4}|z|< |e^z - 1|< \frac{7}{4}|z|$ for $|z|<1$ \cite[Eq. 4.2.38]{milton}. Let $A = \{\xi\in \R^d_+: |\xi| \le x_0 < \frac{1}{|x'-x|}\}.$
\begin{eqnarray}
&&\int_0^t\int_{{\R^d}}[G_{t-r}(x'-y)-G_{t-r}(x-y)]^2dydr\nonumber\\
&=&\frac{1}{(2\pi)^d}\int_0^t\int_{{\R^d}}[G^\ast_{t-r}(\xi)]^2\bigg|1 - e^{i\xi\cdot(x'-x)}\bigg|^2d\xi dr\nonumber\\
&\geq&\frac{1}{4(2\pi)^d}\int_0^t\int_{A}[G^\ast_{t-r}(\xi)]^2|\xi|^2|x'-x|^2 d\xi dr\nonumber\\
&=&c|x'-x|^2\nonumber,
\end{eqnarray}
this is true because by \eqref{uniformbound} $E_\beta(-x)<1$ and using this we have %and using case $d+2<\min(2,\beta^{-1})\a$ and $d+2\geq\min(2,\beta^{-1})\a$ similar to the above proof,
\begin{eqnarray}\int_0^t\int_A[G^\ast_{t-r}(\xi)]^2|\xi|^2d\xi dr&=& \int_0^t\int_{A}(E_\beta\left(-\nu|\xi|^\a(t-r)^\beta\right))^2|\xi|^2d\xi dr \nonumber\\
&\leq&  \int_0^t\int_{A}|\xi|^2d\xi dr \nonumber\\
&\leq& x_{0}^2t|A|,%<\infty %provided $(d+2)\beta/\a<1.
%&\leq&\int_0^t\int_{A}|\xi|^2d\xi dr\nonumber\\
%&\leq&x_0^{2}|A|t,\nonumber
\end{eqnarray}
where $|A|$ is $d-$dimensional volume of $A$.
This completes the proof.
\end{proof}

Our results in this section extend the results in \cite[Chapter 5]{khoshnevisan-cbms} to the time fractional  stochastic heat type equations setting.
Our presentation in this section  follow  the presentation in \cite[Chapter 5]{khoshnevisan-cbms} with many crucial changes.
We next prove that  the linear map $\Phi\to G\circledast\Phi$ is a continuous map from $\mathcal{L}^{\gamma,2}$ into itself. In particular this will show that if $\Phi \in \mathcal{L}^{\gamma,2}$, then $\Psi_t(x):=(G\circledast\Phi)_t(x)$ is also in $\mathcal{L}^{\gamma,2}$, and hence the stochastic convolution $G\circledast\Phi W$ is also a well-defined random field in $\mathcal{L}^{\gamma,2}$.

%Anolog of Theorem 5.1 in Davar's notes
\begin{tm}\label{thm:continuity-stoch-convolution}
If $\Phi \in \mathcal{L}^{\gamma,2}$ for some $\gamma>0$, then $G\circledast\Phi$ has a continuous version that is in $\mathcal{L}^{\gamma,2}$.
\end{tm}

We will prove this theorem in steps.

The following result extends the analogous result on SPDEs by \cite{conus-khoshnevisan,foondun-khoshnevisan-09, mahboubi}.

\begin{prop}\label{younginequality}
(A stochastic Young inequality). For all $\gamma>0, k\in [2,\infty), d<\min\{2, \beta^{-1}\}\alpha$, and $\Phi \in \mathcal{L}^{\gamma,2},$
\begin{equation}\mathcal{N}_{\gamma, k}(G\circledast\Phi)\leq c_0k^{1/2}\cdot \mathcal{N}_{\gamma, k}(\Phi)\nonumber\end{equation}
\\
where $c_0=\sqrt{4(2\gamma)^{-(1-\beta d/\a)}\Gamma(1-\beta d/\a)}$
\end{prop}
\begin{proof}
According to the BDG inequality in \cite[Proposition 4.4]{khoshnevisan-cbms},  and Lemma \ref{Lem:Green1}, for all $k\geq 2,$
\begin{eqnarray}
||(G\circledast \Phi)_t(x)||_k^2&\leq& 4k\int_0^tds\int_{\R^d}dy[G_{t-s}(y-x)]^2||\Phi(y)||_k^2\nonumber\\
&\le&4k[\mathcal{N}_{\gamma,k}(\Phi)]^2\int_0^te^{2\gamma s}ds\int_{\R^d}[G_{t-s}(y-x)]^2 dy\nonumber\\
&=& 4kC^\ast[\mathcal{N}_{\gamma,k}(\Phi)]^2\int_0^t e^{2\gamma s}(t-s)^{-\beta d/\a}ds\nonumber\\
&=&4kC^\ast [\mathcal{N}_{\gamma,k}(\Phi)]^2e^{2\gamma t}\int_0^t e^{-2\gamma u}u^{-\beta d/\a}du\nonumber\\
&\leq& 4kC^\ast [\mathcal{N}_{\gamma,k}(\Phi)]^2e^{2\gamma t}(2\gamma)^{-(1-\beta d/\a)}\Gamma(1-\beta d/\a).
\end{eqnarray}
Divide both sides by $\exp(2\gamma t)$, take supremum over all $(t,x)$, and then take square roots to finish.
\end{proof}
\begin{lem}\label{cont-conv-space}
There exists a finite universal constant $C_1$ such that for all $\gamma>0, k\in [1,\infty), t>0, \Phi\in\mathcal{L}^{\gamma,2}$, $d<\min\{2, \beta^{-1}\}\alpha$ and $x, x'\in\R
^d,$
$$\mathbb{E}\bigg(\big|(G\circledast \Phi)_t(x)-(G\circledast \Phi)_t(x')\big|^k\bigg)\leq (C_1k)^{k/2}e^{\gamma kt}[\mathcal{N}_{\gamma,k}(\Phi)]^k\cdot|x'-x|^{\min\left\{\left(\frac{\a-\beta d}{\beta}\right)^{-},2\right\}\frac k2}.$$
\end{lem}
\begin{proof}
We need to only consider when $\mathcal{N}_{\gamma, k}(\Phi)<\infty,$ a condition that we now assumed. In that case we may apply the BDG inequality to obtain the following
\begin{eqnarray}
||(G\circledast\Phi)_t(x) &-& (G\circledast\Phi)_t(x')||_k^2\nonumber\\
&=&\left\|\int_{(0,t)\times\R^d}[G_{t-s}(y-x)-G_{t-s}(y-x')]\Phi_s(y)W(dsdy)\right\|\nonumber\\
&\leq& c_0k\int_0^tds\int_{\R^d}dy[G_{t-s}(y-x)-G_{t-s}(y-x')]^2\|\Phi_s(y)\|_k^2\nonumber
\end{eqnarray}
%The use of martingale inequalities

Now let us observe that whenever $0<s<t,$
$$\|\Phi_s(y)\|_k^2\leq e^{2\gamma s}[\mathcal{N}_{\gamma,k}(\Phi)]^2\leq e^{2\gamma t}[\mathcal{N}_{\gamma,k}(\Phi)]^2.$$
This and Lemma \ref{DiffEst} gives
\begin{eqnarray}
&&||(G\circledast\Phi)_t(x)- (G\circledast\Phi)_t(x')||_k^2\nonumber\\
&\leq& c_0ke^{2\gamma t}[\mathcal{N}_{\gamma,k}(\Phi)]^2\int_0^tds\int_{\R^d}[G_{t-s}(y-x)-G_{t-s}(y-x')]^2dy\nonumber\\
&\leq&Cc_0ke^{2\gamma t}[\mathcal{N}_{\gamma,k}(\Phi)]^2|x'-x|^{\min\left\{\left(\frac{\a-\beta d}{\beta}\right)^{-},2\right\}}.\nonumber
\end{eqnarray}
Raise both sides to the power of $k/2$ to finish.
\end{proof}

\begin{lem}\label{cont-conv-time}
There exists a finite universal constant $C_1$ such that for every $\gamma >0, k\in[1,\infty],t,t'>0, x\in \R^d$, $d<\min\{2, \beta^{-1}\}\alpha$ and $\Phi\in\mathcal{L}^{\gamma,2},$ then
$$\mathbb{E}\bigg(\big|(G\circledast \Phi)_t(x)-(G\circledast \Phi)_{t'}(x)\big|^k\bigg)\leq (C_1k)^{k/2}e^{\gamma k t}[\mathcal{N}_{\gamma,k}(\Phi)]^k\cdot|t-t'|^{(1-\beta d/\a)k/2}.$$
\end{lem}
\begin{proof}
Without loss of generality, we suppose that $0<t<t'$ and $\mathcal{N}_{\gamma,k}(\Phi)<\infty.$ In that case, we may write
$$(G\circledast \Phi)_t(x)-(G\circledast \Phi)_{t'}(x) = J_1 + J_2,$$
where
$$J_1 := \int_{(0,t)\times\R^d} [G_{t-s}(y-x)-G_{t'-s}(y-x)]\Phi_s(y)W(dsdy),$$
and $$J_2 := \int_{(t,t')\times\R^d} G_{t'-s}(y-x)\Phi_s(y)W(dsdy).$$
We apply the BDG inequality and Lemma \ref{Lem:Green1} and \ref{DiffEst} to obtain the bound
\begin{eqnarray}
\|J_1\|_k^2&\leq&c_0k\int_0^tds\int_{\R^d}dy[G_{t-s}(y-x)-G_{t'-s}(y-x)]^2\|\Phi_s(y)\|_k^2\nonumber\\
&\leq&Cke^{2\gamma t}[\mathcal{N}_{\gamma,k}(\Phi)]^2(t'-t)^{1-\beta d/\a}\nonumber
\end{eqnarray}
and
\begin{eqnarray}
\|J_2\|_k^2&\leq&c_0k\int_t^{t'}ds\int_{\R^d}dy[G_{t'-s}(y-x)]^2\|\Phi_s(y)\|_k^2\nonumber\\
&\leq&Cke^{2\gamma t'}[\mathcal{N}_{\gamma,k}(\Phi)]^2\int_t^{t'}(t'-s)^{-\beta d/\a}ds\nonumber\\
&\leq&Cke^{2\gamma {t'}}[\mathcal{N}_{\gamma,k}(\Phi)]^2(t'-t)^{1-\beta d/\a}.\nonumber
\end{eqnarray}
In this way we obtain the bound
\begin{eqnarray}
\|(G\circledast\Phi)_t(x) - (G\circledast\Phi)_{t'}(x)\|_k&\leq& 2\|J_1\|_k^2 + 2\|J_2\|_k^2\nonumber\\
&\leq& Cke^{2\gamma {t}}[\mathcal{N}_{\gamma,k}(\Phi)]^2(t'-t)^{1-\beta d/\a}\nonumber,
\end{eqnarray}
where $C$ is a finite constant. Hence, we deduce the lemma after we raise both sides of the preceding display to the power $k/2.$
\end{proof}

In the following, we consider
 \begin{equation}\label{Eq:SolL}
 U_t(x)=\int_0^t\int_{\R^d}G_{t-r}(x-y)W(drdy).
 \end{equation}
This is the random part of the mild solution to \eqref{tfspde} when $\sigma\equiv 1$.

\begin{prop}\label{Prop:MomentEst1}
Suppose $d<\min(2,\beta^{-1})\a$, and $k\ge 2$, then we have the following moment estimates for time increments and spatial increments, respectively.

\noindent (i). For $t\le t'$, we have
\begin{equation}\label{Eq:TIncr1}
c^{-1}_{_9}|t'-t|^{\left(1-\frac{\beta d}{\a}\right)\frac k2}\le \E\left[|U_{t'}(x)-U_{t}(x)|^k\right]\le c_{_9}|t'-t|^{\left(1-\frac{\beta d}{\a}\right)\frac k2}.
\end{equation}
(ii). For  $x,\,x'\in{\R^d},$ we have
\begin{equation}\label{Eq:SIncr}
c|x'-x|^k\le\E\left[|U_{t}(x)-U_t(x')|^k\right]\le c_{_{10}}|x-x'|^{\min\left\{\left(\frac{\a-\beta d}{\beta}\right)^{-},2\right\}\frac k2}.
\end{equation}
\end{prop}

\begin{proof}%[Proof of Proposition \ref{Prop:MomentEst1}]
We prove the lower bound in (i) at first. By the H\"older's inequality, the fact that the stochastic integrals over nonoverlapping time intervals are independent, isometry of the stochastic integrals and the proof of Lemma \ref{DiffEst}, we have that for $k\ge 2,$ $t\le t'$
\begin{equation}\label{Eq:TIncrL1}
\begin{split}
&\E\left[|U_{t'}(x)-U_t(x)|^k\right]\ge \left(\E\left[|U_{t'}(x)-U_t(x)|^2\right]\right)^{k/2}\\
&=\left(\E\left|\int_0^t\int_{\R^d}[G_{t'-r}(x-y)-G_{t-r}(x-y)]W(dydr)
+\int_t^{t'}\int_{\R^d}G_{t'-r}(x-y)W(dydr)\right|^2\right)^{k/2}\\
&=\left(\int_0^t\int_{\R^d}[G_{t'-r}(x-y)-G_{t-r}(x-y)]^2dydr+\int_t^{t'}\int_{\R^d}G^2_{t'-r}(x-y)dydr\right)^{k/2}\\
%&\ge\left(\frac{(2^{\beta d/\a}-2)C^\ast(t'- t)^{1-\beta d/\a}}{1-\beta d/\a}+\frac{C^\ast(t'- t)^{1-\beta d/\a}}{1-\beta d/\a}\right)^{p/2}\\
&\ge\left(\frac{C^\ast}{1-\beta d/\a}\right)^{k/2}|t'-t|^{\left(1-\frac{\beta d}{\a}\right)\frac k2}.
\end{split}
\end{equation}
Now, for the upper bound, applying the Burkholder's inequality, the fact that $|a+b|^k\le 2^k\left(|a|^k+|b|^k\right),$ and Lemma \ref{DiffEst} we have that
\begin{equation}\label{Eq:TIncrU}
\begin{split}
&\E\left[|U_{t'}(x)-U_t(x)|^k\right]\\
&\le c_{_{2}}\left(\int_0^t\int_{\R^d}[G_{t'-r}(x-y)-G_{t-r}(x-y)]^2dydr\right)^{k/2}\nonumber\\
&\quad+c_{_{2}}\left(\int_t^{t'}\int_{\R^d}G^2_{t'-r}(x-y)dydr\right)^{k/2}\nonumber\\
&\le c_{2}\left(\frac{C^\ast(t'- t)^{1-\beta d/\a}}{1-\beta d/\a}\right)^{k/2} + c_{2}\left(\frac{C^\ast(t'- t)^{1-\beta d/\a}}{1-\beta d/\a}\right)^{k/2}\nonumber\\
&=2c_{2}\left(\frac{C^\ast}{1-\beta d/\a}\right)^{k/2}|t'-t|^{\left(1-\frac{\beta d}{\a}\right)\frac k2}.
\end{split}
\end{equation}
We now prove (ii). For $p\ge 2$ with $x,\,x'\in\R^d$, by the Burkholder's inequality and Lemma \ref{DiffEst}, we have
\begin{eqnarray}\label{Eq:SIncrU1}
\E\left[|U_{t}(x')-U_{t}(x)|^k\right]
&\le& c_{_{3}}\left(\int_0^t\int_{\R^d}[G_{t-r}(x'-y)-G_{t-r}(x-y)]^2dydr
\right)^{k/2}\nonumber\\
&=& C|x-x'|^{\min\left\{\left(\frac{\a-\beta d}{\beta}\right)^{-},2\right\}\frac k2}.\nonumber
\end{eqnarray}
To prove the lower bound we use the fact that $k\geq 2$ and Lemma \ref{DiffEst}(ii)
\begin{eqnarray}
&\E\left[|U_t(x)-U_t(x')|^k\right]\ge \left(\E\left[|U_t(x)-U_{t}(x')|^2\right]\right)^{k/2}\nonumber\\
&=\left(\int_0^t\int_{\R^d}[G_{t-r}(x'-y)-G_{t-r}(x-y)]^2dydr\right)^{k/2}\nonumber\\
&\geq c|x'-x|^k.\nonumber
\end{eqnarray}

\end{proof}

\begin{proof}[{\bf PROOF OF THEOREM \ref{thm:continuity-stoch-convolution}.}]
If $\Phi\in \mathcal{L}^{\gamma,2}$ for some $\gamma>0$, then $G\circledast \Phi$ is an adapted random field. This can be seen by considering simple random fields $\Phi$, as limits of adapted random fields are adapted. But the result is easy to deduce when $\Phi$ is simple.

Lemma \ref{cont-conv-space} and Lemma \ref{cont-conv-time} shows that $G\circledast \Phi$ is continuous in $L^2(\Omega)$, then  Proposition 4.6 in \cite{khoshnevisan-cbms} implies that $G\circledast \Phi$ is in  $\mathcal{L}^{\gamma,2}$. We give some details of the preceding.

Let us choose and fix $t>0$ and $x\in \rd$. We need to show that $(\tau,z)\to (G^{(t,x)}\circledast \Phi)_\tau(z) $ is continuous in $L^2(\Omega)$ to show that  $G\circledast \Phi$ is in  $\mathcal{L}^{\gamma,2}$.

Since
\begin{equation}\
\begin{split}
 \ & \mathbb{E}\bigg(\big|(G^{(t,x)}\circledast \Phi)_\tau(z)-(G^{(t,x)}\circledast \Phi)_\tau(z')\big|^2\bigg)\\
&\ \ \ \ =\mathbb{E}\bigg(\big|(G^{(t,0)}\circledast \Phi)_\tau(z+x)-(G^{(t,0)}\circledast \Phi)_\tau(z'+x)\big|^2\bigg),
\end{split}
\end{equation}
 the proof of  Lemma \ref{cont-conv-space} ensures that the preceding tends to zero, uniformly on $(\tau, z, z')\in (0,t)\times(-n, n)^d\times (-n, n)^d$ for any $n>0$ fixed, as $|z-z'|\to 0$. This shows uniform continuity in $L^2(\Omega)$, in the space variable.

  On the other hand, we may follow the steps in the proof of Lemma \ref{cont-conv-time}  to deduce that
  \begin{equation}\
\begin{split}
 \ & \mathbb{E}\bigg(\big|(G^{(t,x)}\circledast \Phi)_\tau(z)-(G^{(t,x)}\circledast \Phi)_{\tau'}(z)\big|^2\bigg)\\
&\ \ \ \ =\int_0^{\tau'-\tau}dr\int_{\rd} dy [G_{t-s}(y-z-x)]^2\E\bigg(|\Phi_s(y)|^2\bigg)
\end{split}
\end{equation}
   for $0<\tau<\tau'$. Then we get
    \begin{equation}\
\begin{split}
 \ & \mathbb{E}\bigg(\big|(G^{(t,x)}\circledast \Phi)_\tau(z)-(G^{(t,x)}\circledast \Phi)_{\tau'}(z)\big|^2\bigg)\\
&\ \ \ \ =[\mathcal{N}_{\gamma, 2}(\Phi)]^2\int_0^{\tau'-\tau}e^{2\gamma s}dr\int_{\rd} dy [G_{t-s}(y-z-x)]^2.
\end{split}
\end{equation}
The preceding quantity goes to zero, as $\tau'-\tau \to 0$, uniformly for all $0<\tau<\tau'<t$, and $z\in \rd$. This proves the remaining $L^2$-continuity in the time variable, and completes the proof of the fact that $G\circledast \Phi \in \mathcal{L}^{\gamma,2}$.

The continuity of a modification of  the stochastic convolution $G\circledast \Phi$ follows from Lemmas \ref{cont-conv-space}
 and \ref{cont-conv-time} by using a suitable form of the Kolmogorov continuity theorem \cite[Theorem C.6]{khoshnevisan-cbms}
 \end{proof}
\section{Existence and uniqueness of  solutions to the time fractional SPDEs}\label{convolution}
Recall that $\sigma:\R\to\R$ is Lipchitz continuous. This means that there exists $Lip>0$ such that
\begin{equation}\label{Eq:Cond-sigma}
|\sigma(x)-\sigma(y)|\le Lip|x-y|\quad\mbox{for all }\,\,x,\,y\in\R.
\end{equation}
We may assume, without loss of generality, that $Lip$ is also greater than $|\sigma(0)|$. Since $|\sigma(x)|\le |\sigma(0)|+Lip |x|$, it follows that $|\sigma(x)|\le Lip (1+|x|)$ for all $x\in \R$.

The next theorem establishes existence and uniqueness of  mild solutions of \eqref{tfspde}. It is our first main result in this paper.
%Analog of Theorem 5.5. in Davar's notes!
\begin{tm}\label{thm:existence-moment}
Let $d<\min \{2, \beta^{-1}\}\alpha$.
If $\sigma$ is Lipschitz continuous and $u_0$ is measurable and bounded, then there exists a continuous random field $u\in \cup_{\gamma>0}\mathcal{L}^{\gamma,2}$ that solves \eqref{tfspde} with initial function $u_0$. Moreover, $u$ is a.s.-finite among all random fields that satisfy the following: There exists a positive and finite constant $L$-depending only on $Lip,$ and $\sup_{z\in \rd}|u_0(z)|$- such that
\begin{equation}\label{uniform-moment-bound}
\sup_{x\in \rd} \E\bigg(|u_t(x)|^k\bigg)\leq L^k\exp(Lk^{1+\alpha/(\alpha-\beta d)} t).
\end{equation}
\end{tm}
\begin{rk}
Theorem \ref{thm:existence-moment} implies that a random field solution exists when  $d<\min \{2, \beta^{-1}\}\alpha$. So in the case $\alpha=2, \beta<1/2$, a random field solution exists when $d=1,2,3$.
 The analogous result in the case $\beta=1$ is only valid when $d=1$.
\end{rk}
\begin{rk}
Is the power $k^{1+\alpha/(\alpha-\beta d)}$  in Theorem \ref{thm:existence-moment} artificial? Is it possible to find a lower bound for the moments with the same power of $k$?
\end{rk}

We use Picard iteration to prove Theorem \ref{thm:existence-moment} that is outlined in \cite[Chapter 1]{khoshnevisan-cbms} with crucial changes. Define $u_t^{(0)}(x):=u_0(x)$, and iteratively define $u_t^{(n+1)}$ from $u_t^{(n)}$ as follows:
\begin{equation}\label{Eq:Picard1}
\begin{split}
u_t^{(n+1)}(x)&:=(G_t*u_0)(x)+\int_{(0,t)\times \R^d}G_{t-r}(x-y)\sigma(u^{(n)}(r,y))W(drdy)\\
&:=(G_t*u_0)(x)+(G\circledast\sigma\bigg(u^{(n)}\bigg)W)_t(x)
\end{split}
\end{equation}
for all $n\geq 0, t>0$, and $x\in \rd$. Moreover, we set $u_0^{(k)}(x):=u_0(x)$ for every $k\geq 1$ and $x\in \rd$.

%analog of Proposition 5.8 in Davar's lecture notes!
\begin{prop}\label{prop:picard}
The random fields $\{u^{(n+1)}\}_{n=0}^\infty$ are well-defined, and each is in $\cup_{\gamma>0}\mathcal{L}^{\gamma,2}$. Moreover, there exist positive and finite constants $L_1$ and $L_2$--depending only on $Lip$, and $\sup_{z\in \rd}|u_0(z)|$--such that
\begin{equation}\label{uniform-moment-bound-picard}
\sup_{x\in \rd} \E\bigg(|u_t^{(n)}(x)|^k\bigg)\leq L_1^k\exp(L_2k^{1+\alpha/(\alpha-\beta d)} t),
\end{equation}
simultaneously for all $k\in [1,\infty)$, $n\geq 0$, and $t>0$.
\end{prop}

\begin{proof}
We prove this proposition using induction on $n$. Since $u_0$ is non-random, bounded, and measurable, it is in $\cup_{\gamma>0}\mathcal{L}^{\gamma,2}$. Moreover \eqref{uniform-moment-bound-picard} holds since

$$
\sup_{x\in \rd}\E\bigg(\big|u_t^{(0)}(x)\big|^k\bigg)=\sup_{z\in \rd}|u_0(z)|\leq L_1^k\exp(L_2k^{1+\alpha/(\alpha-\beta d)} t).
$$

Next suppose that the proposition holds  for some integer $n\ge 0$. We will show the proposition for $n+1$.

Let $k$ denote a fixed real number $\ge 2$. Now
$$
u_t^{(l+1)}(x):=(G_t*u_0)(x)+B_t^{(l)}(x),\ \ \mathrm{where}:
$$

$$
B_t^{(l)}(x):=\int_{(0,t)\times \R^d}G_{t-r}(x-y)\sigma(u^{(l)}(r,y))W(drdy),
$$
for all $t>0, x\in \rd$, and $l\ge 0$.

It is easy to see that $|(G_t*u_0)(x)|\leq \sup_{z\in \rd}|u_0(z)|$, hence for all $\gamma>0$
\begin{equation}\label{a-ineq}
\mathcal{N}_{\gamma, k}((G*u_0))\leq \sup_{z\in \rd}|u_0(z)|.
\end{equation}
We estimate $B^{(n)}$ using the Stochastic Young inequality in Proposition \ref{younginequality}:

\begin{equation}\label{bn-ineq}
    \begin{split}
    \mathcal{N}_{\gamma, k}\bigg(B^{(n)}\bigg)&\leq c_{\alpha, \beta, d}\frac{\sqrt{k}}{\gamma^{1-\beta d/\alpha}}\mathcal{N}_{\gamma, k}\bigg(\sigma(u^{(n)})\bigg)\\
    &\leq c_{\alpha, \beta, d}\cdot Lip \cdot \frac{\sqrt{k}}{\sqrt{\gamma^{1-\beta d/\alpha}}}\cdot \bigg\{1+\mathcal{N}_{\gamma, k}\bigg(u^{(n)}\bigg)\bigg\},
    \end{split}
\end{equation}
where $c_{\alpha, \beta, d}$ is the constant in Proposition \ref{younginequality}.

Now combining equations \eqref{a-ineq} and \eqref{bn-ineq} we get for all $\gamma>0$
\begin{equation}\label{u-n-ineq}
    \begin{split}
    \mathcal{N}_{\gamma, k}\bigg(u^{(n+1)}\bigg)& \leq \sup_{z\in \rd}|u_0(z)|+ c_{\alpha, \beta, d}\cdot Lip \cdot \frac{\sqrt{k}}{\sqrt{\gamma^{1-\beta d/\alpha}}}\\
    &+ c_{\alpha, \beta, d}\cdot Lip \cdot \frac{\sqrt{k}}{\sqrt{\gamma^{1-\beta d/\alpha}}}\cdot \mathcal{N}_{\gamma, k}\bigg(u^{(n)}\bigg).
    \end{split}
\end{equation}
We can find a constant $L_2>0$--depending only on $Lip$, $\alpha, \beta, d$--such that $\gamma:=L_2 k^{\alpha/(\alpha-\beta d)}$ satisfies
$$
c_{\alpha, \beta, d}\cdot Lip \cdot \frac{\sqrt{k}}{\sqrt{\gamma^{1-\beta d/\alpha}}}\leq \frac14,
$$
then
$$
\mathcal{N}_{\gamma, k}\bigg(u^{(n+1)}\bigg)\leq \sup_{z\in \rd}|u_0(z)|+\frac14+\frac14\mathcal{N}_{\gamma, k}\bigg(u^{(n)}\bigg).
$$

In order to simplify notation, let
$$
\theta=\sup_{z\in \rd}|u_0(z)|+\frac14.
$$

We solve recursively to find that

\begin{eqnarray*}
\mathcal{N}_{\gamma, k}\bigg(u^{(n+1)}\bigg)&\leq& \theta + \frac\theta4 +\frac{1}{4^2}\mathcal{N}_{\gamma, k}\bigg(u^{(n-1)}\bigg)\\
&&\le \cdots\le \sum_{j=1}^n\frac{\theta}{4^j}+ \frac{1}{4^{n+1}}\mathcal{N}_{\gamma, k}\bigg(u^{(0)}\bigg).
\end{eqnarray*}
But
$$
\mathcal{N}_{\gamma, k}\bigg(u^{(0)}\bigg)=\sup_{t\geq 0}\sup_{x\in\rd}\bigg(e^{-\gamma t}|u_0(x)|\bigg)=\sup_{z\in \rd}|u_0(z)|.
$$
Thus, for our choice of $\gamma$, we have
\begin{eqnarray*}
\mathcal{N}_{\gamma, k}\bigg(u^{(n+1)}\bigg)&\leq& \frac43 \theta+ 4^{-(n+1)} \sup_{z\in \rd}|u_0(z)|\leq \frac43 \theta + \sup_{z\in \rd}|u_0(z)|=L_1.
\end{eqnarray*}

Therefore by Theorem \ref{thm:continuity-stoch-convolution} we get $u^{(n+1)}\in \cup_{\gamma>0}\mathcal{L}^{\gamma,2}$. Moreover, the last inequality means:
$$
\E\bigg(\bigg|u_t^{(n+1)}(x)\bigg|^k\bigg)L_1^k e^{\gamma k t}=L_1^k\exp(L_2k^{1+\alpha/(\alpha-\beta d)} t).
$$
Hence, \eqref{uniform-moment-bound-picard}  holds for $n+1$, and this proves the induction step, and proves the proposition.
\end{proof}

\begin{proof}[{\bf PROOF OF THEOREM \ref{thm:existence-moment}.}]
Our proof follows similar steps as in the proof of Theorem 5.5 in \cite{khoshnevisan-cbms} with nontrivial crucial changes.
Let's choose and fix some $k\in [2,\infty)$.
Let us write
\begin{eqnarray*}
J&:=& u_t^{(n+1)}(x)-u_t^{(n)}(x)\\
&=&\int_{(0,t)\times \R^d}G_{t-r}(x-y)\bigg[\sigma(u^{(n)}(r,y))- \sigma(u^{(n-1)}(r,y))\bigg]W(drdy).
\end{eqnarray*}
By proposition \ref{prop:picard} every $u^{(n)}\in \mathcal{L}^{\gamma,2}$ for some $\gamma>0$, and hence $J$ is a well-defined stochastic integral for every $n\geq 0$. By the same proposition we can also choose a continuous version of $(t,x)\to u_t^{(n)}(x)$ for every $n$.

Next we estimate $J$. We apply the BDG inequality (Propositon 4.4 in \cite{khoshnevisan-cbms})to bound $\|J\|_k$ as follows:
\begin{eqnarray*}
\|J\|_k^2&\leq &C\cdot k \int_0^t dr\int_\rd dy [G_{t-r}(x-y)]^2\bigg\|\sigma(u^{(n)}_r(y))- \sigma(u^{(n-1)}_r(y))\bigg\|_k^2\\
&\leq & C\cdot  Lip^2k \int_0^t dr\int_\rd dy [G_{t-r}(x-y)]^2\bigg\|u^{(n)}_r(y)- u^{(n-1)}_r(y)\bigg\|_k^2\\
&\leq & C\cdot  Lip^2k\bigg[\mathcal{N}_{\gamma, k}\bigg(u^{(n)}- u^{(n-1)}\bigg)\bigg]^2 \int_0^te^{2\gamma r} dr\int_\rd dy [G_{t-r}(x-y)]^2
\end{eqnarray*}
Since  $\int_\rd dy [G_{t-r}(x-y)]^2=C^*(t-r)^{-\beta d/\alpha}$, it follows that
\begin{eqnarray*}
\|J\|_k^2
&\leq & C\cdot  Lip^2k e^{2\gamma t}\bigg[\mathcal{N}_{\gamma, k}\bigg(u^{(n)}- u^{(n-1)}\bigg)\bigg]^2 \int_0^t \frac{e^{-2\gamma r}}{r^{\beta d/\alpha}} dr\\
&\leq & \frac{C\cdot  Lip^2k}{\gamma^{1-\beta d/\alpha}} e^{2\gamma t}\bigg[\mathcal{N}_{\gamma, k}\bigg(u^{(n)}- u^{(n-1)}\bigg)\bigg]^2 .
\end{eqnarray*}
Since the right hand side does not depend on $(t,x)$ after dividing by $e^{2\gamma t}$, we optimize over $(t,x)$ to find the recursive inequality,
$$
\mathcal{N}_{\gamma, k}\bigg(u^{(n+1)}- u^{(n)}\bigg)\leq C\cdot Lip \mathcal{N}_{\gamma, k}\bigg(u^{(n)}- u^{(n-1)}\bigg)\bigg[\frac{\sqrt{k}}{\sqrt{\gamma^{1-\beta d/\alpha}}}\bigg].
$$

The preceding holds for all $\gamma>0$. We can choose $\gamma_0:=qk^{\alpha/(\alpha-\beta d)}$, where $q> L_2$ depends only on $L_2$ and ensures that
$$
\mathcal{N}_{\gamma_0, k}\bigg(u^{(n+1)}- u^{(n)}\bigg)\leq \frac14 \mathcal{N}_{\gamma_0, k}\bigg(u^{(n)}- u^{(n-1)}\bigg).
$$
According to Proposition \ref{prop:picard}
$$
\mathcal{N}_{\gamma_0, k}\bigg(u^{(1)}- u^{(0)}\bigg)\leq \mathcal{N}_{\gamma_0, k}\bigg(u^{(1)}\bigg)+\mathcal{N}_{\gamma_0, k}\bigg(u^{(0)}\bigg).
$$
Since $q>L_2$, Proposition \ref{prop:picard} implies that
$$
\mathcal{N}_{\gamma_0, k}\bigg(u^{(1)}\bigg)\vee \mathcal{N}_{\gamma_0, k}\bigg(u^{(0)}\bigg)\le L_1.
$$
Hence we obtain the estimate
$$
\mathcal{N}_{\gamma_0, k}\bigg(u^{(n+1)}- u^{(n)}\bigg)\leq \frac{L_1}{4^n},
$$
valid for all $n\geq 0$ and $k\in [2,\infty)$.
From this we obtain
\begin{enumerate}
\item The random field $u:=\lim_{n\to\infty}u^{(n)}$ exists, where the limit takes place almost surely and in every norm $\mathcal{N}_{\gamma_0, k}$;
\item the random field $\mathbb{S}$ defined by
$$
\mathbb{S}_t(x):=\lim_{n\to\infty} \int_{(0,t)\times \R^d}G_{t-r}(x-y)\sigma(u^{(n)}(r,y))W(drdy),$$
exists, where the limit takes place almost surely and in every norm $\mathcal{N}_{\gamma_0, k}$.
\end{enumerate}

Combining Lemmas \ref{cont-conv-space} and \ref{cont-conv-time}, applied to $\Phi:=u^{(n)}$,  with Proposition \ref{uniform-moment-bound-picard} we see that for every $k\in [2,\infty)$ and  $\tau\in (0,\infty)$ there exists a finite constant $A_{k,\tau}$ such that
$$
\E\bigg(\big|u_t^{(n)}(x)-u_{t'}^{(n)}(x')\big|^k\bigg)\leq  A_{k,\tau}\bigg(|x-x'|^{\min\left\{\left(\frac{\a-\beta d}{\beta}\right)^{-},2\right\}\frac k2}+|t-t'|^{\frac{k(1-\beta d/\alpha)}{2}}\bigg),
$$
simultaneously for all $t,t'\in [0,\tau]$, $x,x'\in \rd.$ The right hand side of this inequality does not depend on $n$. Hence, using Fatou's lemma  we get
\begin{equation}
\E\bigg(\big|u_t(x)-u_{t'}(x')\big|^k\bigg)\leq  A_{k,\tau}\bigg(|x-x'|^{\min\left\{\left(\frac{\a-\beta d}{\beta}\right)^{-},2\right\}\frac k2}+|t-t'|^{\frac{k(1-\beta d/\alpha)}{2}}\bigg).
\end{equation}

 Now by a suitable form of Kolmogorov continuity theorem in \cite[Theorem c.6]{khoshnevisan-cbms} we get that $u$ has a version that is continuous. Moreover, \eqref{uniform-moment-bound} holds with $L=\max\{q, L_1\}$ since
 $\mathcal{N}_{\gamma_0, k}(u)\leq L_1$ for all $k\in [2,\infty)$.

 So far, we know that
 \begin{equation}\label{eq-limit}
    u_t(x)=(G_t*u_0)(x)+\mathbb{S}_t(x),
 \end{equation}
 where the equality is understood in the sense that the $\mathcal{N}_{\gamma_0, k}$-norm of the difference between the two sides of that inequality is zero. Equivalently,  by the use of Fubini theorem we have shown that with probability one, the identity \eqref{eq-limit} holds for almost every $t>0$ and $x\in \rd$.

 Since $u=\lim_{n\to\infty} u^{(n)}$ and $u^{(n)}\in \mathcal{L}^{\gamma, 2}$ for some $\gamma>0$ that is independent of $n$, we can conclude that $u\in \mathcal{L}^{\gamma, 2}$, and hence
 $$
 \bar{\mathbb{S}}_t(x):= \int_{(0,t)\times \R^d}G_{t-r}(x-y)\sigma(u(r,y))W(drdy),
 $$
 is well-defined for all $t>0$, and $x\in \rd$. Now we apply the Fatou's Lemma together with Proposition \ref{younginequality} with $\Phi=u^{(n)}-u$ and $\gamma_0=qk^{\alpha/(\alpha-\beta d)}$ to get that
 $$
 \mathcal{N}_{\gamma_0, k}(\mathbb{S}-\bar{\mathbb{S}})\leq C\cdot\liminf_{n\to\infty}\mathcal{N}_{\gamma_0, k}\bigg(u^{(n)}- u\bigg)=0.
 $$
 Hence $\mathbb{S}$ and $\bar{\mathbb{S}}$ are versions of one another. Another application of Lemmas \ref{cont-conv-space} and \ref{cont-conv-time} shows that  $\mathbb{S}$  has a continuous version. We combine these results with \eqref{eq-limit} to see that the present version of $u$ is a mild solution--in the sense of   \eqref{Eq:Mild} for the right versions of the integrals--for the time  fractional  stochastic heat type equation \eqref{tfspde}. This completes the proof of existence.

 Suppose $v\in \mathcal{L}^{\gamma, 2}$ is another random field that is mild solution to \eqref{tfspde} with initial function $u_0$. We can argue  similar to the arguments above to show that if $Q$ is sufficiently large, $\gamma_1=Qk^{\alpha/(\alpha-\beta d)}$, then
 $
\mathcal{N}_{\gamma_1, k}(u-v)\leq \frac14 \mathcal{N}_{\gamma_1, k}(u-v).
$
Hence $u$ and $v$ are versions of one another. The other details are omitted.

\end{proof}
%\subsection{Mild implies weak}
%{\color{red}Can we do this section that is the section 5.3 in Davar's notes. I will work on this!}

\section{Non-existence of solutions}

%{\color{red}How about the results here. I am working on this! The result here are corresponding to Section 5.4 in Davar's notes.}

 In this section, we will establish the non-existence of  finite energy solutions when $\sigma $ grows faster than linear.
 We say that a random field $u$ is a finite energy solution to the stochastic heat equation \eqref{tfspde} when $u\in \cup_{\gamma>0}\mathcal{L}^{\gamma, 2}$ and there exists $\rho_*>0$ such that
 $$
 \int_0^\infty e^{-\rho_* t}\E(|u_t(x)|^2)dt<\infty\ \ \ \mathrm{for\ all}\ \ x\in \rd.
 $$

 \begin{rk}
 If $\rho\in(0,\infty)$, then
 $$
 \int_0^\infty e^{-\rho t}\E(|u_t(x)|^2)dt\leq [ \mathcal{N}_{\gamma,2}(u)]^2\cdot \int_0^\infty e^{-(\rho-2\gamma)t}dt.
 $$
 Therefore if $\rho>2\gamma$ and $\mathcal{N}_{\gamma,2}(u)<\infty$, then the preceding integral is finite.  By Theorem \ref{thm:existence-moment}, when $\sigma $ is Lipschitz-continuous function and $u_0$ is bounded and measurable, then there exists a finite energy solution to the time fractional stochastic  type equation \eqref{tfspde}.
 \end{rk}

 When $\sigma$ is Lipschitz continuous then it is at most linear growth. The following theorem shows that if we drop the assumption of linear growth, then we do not have a finite energy solution to the time fractional stochastic heat equation \eqref{tfspde}. This theorem extends the result of Foondun and Parshad \cite{foondun-parshad}
 \begin{tm}
 Suppose $\inf_{z\in \rd}u_0(z)>0$ and $\inf_{y\in \rd}|\sigma(y)|/|y|^{1+\epsilon}>0.$ Then, there is no finite-energy solution to the time fractional stochastic  heat type equation \eqref{tfspde}.
 \end{tm}

 \begin{proof}
 We will adapt the methods in  \cite{foondun-parshad} with many crucial changes. Let $c:=\inf_{y\in \rd}|\sigma(y)|/|y|^{1+\epsilon}$ and $l:=\inf_{z\in \rd}u_0(z)$.

% Suppose, on the contrary, that there exists a finite-energy solution $u$ to \eqref{tfspde}.
 \begin{eqnarray}\label{supersolution}
\E(|u_t(x)|^2) &=& |(G_t\ast u_0)(x)|^2 + \int_0^tds\int_{{\R^d}}dy[G_{t-s}(y-x)]^2\E(\sigma^2(u_s(y))\nonumber\\
&\geq&l^2 + c^2\int_0^t\int_{{\R^d}}[G_{t-s}(y-x)]^2\E(|u_s(y)|^{2(1+\epsilon)}) dy\\
&\geq&l^2 + c^2\int_0^t\int_{{\R^d}}[G_{t-s}(y-x)]^2\inf_{y\in \R^d}[\E(|u_s(y)|^{2})]^{1+\epsilon} dy,
\end{eqnarray}
 with an application of the Jensen's inequality in the last inequality.
 If we let
 $$
 I(t) := \mbox{inf}_{y\in {\R^d}}\E(|u_s(t)(y)|^2) \ \ (t\geq 0),
 $$
 it satisfies the following inequality
 \begin{equation}\label{superforI}
 \begin{split}
 I(t)&\geq l^2+c^2\int_0^t||G_{t-s}||^2_{L^2(\rd)}[I(s)]^{1+\epsilon}ds\\
 &=l^2+c^2\int_0^t (t-s)^{-d\beta/\a}[I(s)]^{1+\epsilon}ds.
  \end{split}
 \end{equation}
  Now we define the Laplace transform of $I$ as
 $$
 \tilde{I}(\theta):=\int_0^\infty e^{-\theta s}I(s)ds.
 $$
% Since $u$ is a finite energy solution by assumption, we certainly have $\tilde{I}(\rho_*)<\infty$ for some $\rho_*>0$. Choose and fix such a $\rho_*$.
Because
 $$
 \int_0^\infty e^{-\theta s}||G_{s}||^2_{L^2(\rd)}=\int_0^\infty e^{-\theta s}C^*s^{-\beta d/\alpha} ds=C_1 \theta^{-(1-\beta d/\alpha)}
 $$
 for all $\theta>0$, where $C_1=C^* \Gamma(1-\beta d/\alpha)$.
 It follows that
 \begin{equation}
 \tilde{I}(\theta)\geq \frac{l^2}{\theta}+c^2 C_1 \theta^{-(1-\beta d/\alpha)} \int_0^\infty e^{-\theta s}[I(s)]^{1+\epsilon}ds.
 \end{equation}
  We multiply both sides by $\theta$ and use Jensen's inequality to find that

  \begin{eqnarray*}
  \theta\tilde{I}(\theta)&\geq& l^2+c^2 C_1 \theta^{-(1-\beta d/\alpha)}\bigg[ \int_0^\infty \theta e^{-\theta s}[I(s)]ds\bigg]^{1+\epsilon}\\
  &=& l^2+c^2 C_1 \theta^{-(1-\beta d/\alpha)}\big[\theta \tilde{I}(\theta)\big]^{1+\epsilon},
  \end{eqnarray*}
  for all $\theta >0$. It follows that $ \theta\tilde{I}(\theta)\geq l^2>0$, and hence $\tilde{I}(\theta)>0$ for all $\theta>0$. Moreover, $[\theta\tilde{I}(\theta)]^{1+\epsilon}\geq l^{2\epsilon}\theta\tilde{I}(\theta),$ and hence
  $$
  \theta\tilde{I}(\theta)\geq l^2+c^2 C_1 \theta^{-(1-\beta d/\alpha)}l^{2\epsilon}\theta\tilde{I}(\theta).
  $$
   We shall show that
   $$\tilde{I}(\theta) = \infty \ \ \text{for}\ \ 0 < \theta \leq \theta_{0}:=(c^2 C_1 l^{2\epsilon})^{\alpha/(\alpha-\beta d)}.$$
  %This implies that $\theta\tilde{I}(\theta)=\infty$ when
Under the assumption that $0 < \theta \leq \theta_{0}$%(c^2 C_1 l^{2\epsilon})^{\alpha/(\alpha-\beta d)}$
, the constant $A := c^2 C_1 \theta^{-(1-\beta d/\alpha)}l^{2\epsilon}$ is greater than or equal to $1$. With  recursive argument we can show that for any positive integer n,
  $$\theta\tilde{I}(\theta)\geq l^2 + A^n\theta\tilde{I}(\theta).$$
Since $A\geq 1, l >0 $ and $n$ is arbitrary. We see that $\theta\tilde{I}(\theta) = \infty.$  Which shows that $\tilde{I}(\theta) = \infty$ for a given range of $\theta.$ %Suppose that there exists a $\theta_0<\infty$  which is the minimum of all $\theta$ such that $\tilde{I}(\theta) < \infty$ for all $\theta > \theta_0$. But this contradicts \eqref{superforI}. Hence, there is no finite-energy solution.

Now we will show that there exists $\theta_1>\theta_0$ such that $\tilde{I}(\theta_1) = \infty.$ Recall that $\tilde{I}(\theta_0) = \infty$. That means $I(t) \geq ce^{\theta_0 t}$ for large $t$.

For $t$ large enough, we have
\begin{eqnarray}
I(t)&\geq& l^2 + c^2\int_0^t (t-s)^{-d\beta/\a}[I(s)]^{1+\epsilon}ds\nonumber\\
&\geq& l^2 + c^2t^{-d\beta/\a}\int_0^t [I(s)]^{1+\epsilon}ds\nonumber\\
&\geq&l^2 + c^2t^{-d\beta/\a}\int_{t/2}^t e^{\theta_0(1+\epsilon)s}ds\nonumber\\
&\geq&l^2 + c^2t^{-d\beta/\a} e^{\theta_0(1+\epsilon)t}\left(1-e^{-\theta_0(1+\epsilon)t/2}\right).
\end{eqnarray}
Using the inequality $x/(1+x) < (1-e^{-x})$ for $x>-1$ we get
$$I(t)\geq l^2 + c^2t^{-d\beta/\a} e^{\theta_0(1+\epsilon)t},$$
again for large enough $t$. Now if we take $\theta_1 = \theta_0(1+\epsilon)$, we have $\tilde{I}(\theta_1) = \infty.$ Since we can repeat this process over and over again there is no minimum $\theta < \infty$ such that $\tilde{I}(\theta) = \infty.$

{    Now, if we assume there is a finite energy solution, we certainly have $\tilde{I}(\theta_\ast)<\infty$ for some $\theta_\ast > 0.$ With simple algebra we can show that $$\tilde{I}(\theta)= \int_0^\infty e^{-\theta s}I(s)ds=\int_0^\infty e^{-(\theta-\theta_\ast + \theta_\ast)s}I(s)ds\leq \int_0^\infty e^{-\theta_\ast s}I(s)ds<\infty,$$ for $\theta\geq\theta_\ast$. That means $\tilde{I}(\theta)<\infty$ for all $\theta \geq \theta_\ast.$ But this contradicts the above argument. Hence there is no finite energy solution.} This completes the proof.

 \end{proof}

 \section{Equivalence of time fractional SPDEs to higher order SPDE's}
In this section we will give a connection of time fractional SPDE's and Higher order parabolic SPDE's.

%{\color{red} I will work on this Part-Erkan!}
%{Parabolic SPDE's}

 Allouba \cite{allouba-spde1, allouba-spde2} considered the following SPDE($k=1,2,\cdots$):
\begin{equation}\label{spde-order-m}
\begin{split}
\partial_t u^k_t(x) & =\sum_{j=1}^{2^k-1} \frac{t^{j/2^k-1}}{\Gamma(j/2^k)} \Delta_x^j u_0^k(x)
+ \Delta_x^{2^k} u^k_t(x) +\sigma(u^k)\stackrel{\cdot}{W}(t,x);\\
 u^k_t(x)|_{t=0} &=  u^k_0(x)
\end{split}
\end{equation}
Where $\sigma$ satisfy Lipschitz and linear growth conditions.
Allouba \cite{allouba-spde2}  showed that there exists a path wise unique strong (mild) solution in all space dimensions $d=1,2,3$ given by
\begin{equation}
\begin{split}
u^k_t(x)&=\int_\R^d K^{BM^d,E^\beta}_t(x-y)u^k_0(y)dy\\
      &+\int_0^t\int_\R  K^{BM^d,E^\beta}_{t-r}(x-y)\sigma(u^k_s(y))W(dydr)
\end{split}
\end{equation}
where $K^{BM^d,E^\beta}_t(x-y)$ is the transition density of a $d$-dimensional process $B(E^\beta)$, here $B$ is a $d$-dimensional Brownian motion,
and $E^\beta$ is inverse of a stable subordinator of index $\beta=2^{-k}$:
$$
K^{BM^d,E^\beta}_t(x-y)=\int_{0}^\infty \frac{e^{-\frac{||x-y||^2}{4s}}}{\sqrt{4\pi s}}f_{E_t}(s)ds.
$$

Allouba  \cite{allouba-spde1, allouba-spde2} also showed that the  H\"older  continuity exponent (time, space)  of the mild solution of \eqref{spde-order-m} is
$$\left(\left(\frac{2\beta^{-1}-d}{4\beta^{-1}}\right)^{-}, \left(\frac{4-d}{2}\wedge 1\right)^{-}\right),$$
  (where $\beta=2^{-k}$) in space dimensions $d=1, 2,3$.

Since  Baeumer et al.\cite{bmn-07} showed that $G_t(x-y)=K^{BM^d,E^\beta}_t(x-y)$ for $\alpha=2$ and
$$
K^{Y,E^\beta}_t(x-y)=\int_{0}^\infty p_{Y(s)}(x-y)f_{E_t}(s)ds =G_{t}(x-y),\ \ \mathrm{for}\ \ 0<\alpha<2,
$$
we obtain the next result
\begin{tm}\label{thm:equivalence}
Let $L_x=-\nu(-\Delta)^{\alpha/2}$  for $\alpha\in (0,2])$ and suppose $\sigma$ is Lipschitz.
 For any $k=1,2,3,\ldots$ both the higher order SPDE
{
\begin{equation}\label{spde-order-m1}
\begin{split}
\partial_t u_t^k(x) & =\sum_{j=1}^{2^k-1} \frac{t^{j/2^k-1}}{\Gamma(j/2^k)} L_x^j u^k_0(x)
+ L_x^{2^k} u_t^k(x) +\sigma(u_k)\stackrel{\cdot}{W}(t,x);\\
 u^k_t(x)|_{t=0} &=  u^k_0(x),
\end{split}
\end{equation}
}
and the time fractional  SPDE
\begin{equation}\label{tfspde-11}
 \partial^{1/2^k}_tu^k_t(x)=L_x u^k_t(x)+I^{1-1/2^k}_t[\sigma(u^k)\stackrel{\cdot}{W}(t,x)];\ \   u^k_t(x)|_{t=0} =  u^k_0(x),
 \end{equation}
have the same unique mild solution given by
{
\begin{equation}\begin{split}
u^k_t(x)&=\int_\R^d K^{Y,E^\beta}_t(x-y)u^k_0(y)dy\\
      &+\int_0^t\int_\R  K^{Y,E^\beta}_{t-r}(x-y)\sigma(u^k_s(y))W(dydr).
\end{split}
\end{equation}
}

\end{tm}

\begin{rk}
Theorem \ref{thm:equivalence} extends the particular case considered by Baeumer et al. \cite{bmn-07}--the case of $\sigma\equiv 0$-- to the equivalence of SPDE's.
Allouba \cite[equation (1.8)]{allouba-spde2} mentions equivalence of \eqref{spde-order-m} to the equation
\begin{equation}\label{tfspde-12}
 \partial^{1/2^k}_tu^k_t(x)=\Delta_x u^k_t(x)+\sigma(u^k)\stackrel{\cdot}{W}(t,x);\ \   u^k_t(x)|_{t=0} =  u^k_0(x).
 \end{equation}
Our theorem  \ref{thm:equivalence} gives the correct equivalence.
\end{rk}

\end{document}